\documentclass[11pt]{article}

\usepackage[a4paper,margin=1in]{geometry}
\usepackage{amsmath,amssymb,amsthm,mathtools}
\usepackage{bm}
\usepackage{hyperref}
\usepackage{enumitem}
\usepackage{xcolor}
\usepackage{tikz}
\usetikzlibrary{matrix,positioning,calc}

\newtheorem{theorem}{Theorem}
\newtheorem{lemma}[theorem]{Lemma}
\newtheorem{corollary}[theorem]{Corollary}

\theoremstyle{definition}
\newtheorem{definition}[theorem]{Definition}
\newtheorem{assumption}[theorem]{Assumption}
\newtheorem{remark}[theorem]{Remark}
\newtheorem{example}[theorem]{Example}

\newcommand{\C}{\mathbb{C}}
\newcommand{\Z}{\mathbb{Z}}
\newcommand{\R}{\mathbb{R}}

\newcommand{\ellTwo}{\ell^2}
\newcommand{\norm}[1]{\left\|#1\right\|}

\newcommand{\supp}{\operatorname{supp}}
\newcommand{\dist}{\operatorname{dist}}
\newcommand{\Lip}{\operatorname{Lip}}

\hypersetup{
  hidelinks,
  pdftitle={Spectral and Pseudospectral Approximation of
            Finite-Interaction-Range Operators in Doubling Metric Measure Spaces},
  pdfauthor={Mattes Wittig},
  pdfsubject={Spectral theory},
  pdfkeywords={pseudospectrum, finite sections, band operators,
               doubling metric measure spaces}
}

\newcommand{\nuop}{\nu}

\title{Spectral and Pseudospectral Approximation of Finite-Interaction-Range Operators in Doubling Metric Measure Spaces}
\author{Mattes Wittig\thanks{
  Hamburg University of Technology (TUHH), Institute of Mathematics, Am Schwarzenberg-Campus 3, D-21073 Hamburg, Germany, \href{mailto:mattes.wittig@tuhh.de}{mattes.wittig@tuhh.de}
}}
\date{}

\begin{document}
\maketitle

\begin{abstract}
We study bounded finite-interaction-range operators $H$ on $\ell^2(\Gamma)$, where $\Gamma$ is a uniformly discrete subset of a left-invariant doubling metric measure space $(X,d,\mu)$. Our goal is to approximate spectral information of $H$ from \emph{finite sections} $H_{L,x}$ supported on balls $B_L(x)\cap\Gamma$. The main technical input is a commutator estimate for Lipschitz ``tent'' localisations $W_{L,x}$, which depends only on geometric properties (doubling) and a uniform interaction degree. As a consequence, we obtain explicit two-sided pseudospectral inclusion bounds of the form
\[
  \gamma_{L,\varepsilon}(H)\subset \sigma_\varepsilon(H)\subset \gamma_{L,\varepsilon+C_0/L}(H),
\]
and Hausdorff convergence of window pseudospectra to the (global) pseudospectrum as $L\to\infty$. In the self-adjoint/ normal case this yields computable gap tests and spectral sampling schemes with rigorous $O(1/L)$ error control, while in the non-normal case it leads to corresponding approximation results for pseudospectra. The framework isolates the geometric core behind earlier approaches on $\R^n$ and on countable Abelian groups, and covers irregular geometries arising in quasicrystal models, as well as new cases such as discrete (non-Abelian) nilpotent groups (including for example the discrete Heisenberg group).
\end{abstract}

{\small

\smallskip

}

\section{Introduction and motivation}

Many operators arising in mathematical physics, solid state theory, and aperiodic
order act on an infinite set $\Gamma$, such as lattices, Delone sets, or model
sets. Typical examples include discrete Schr\"odinger operators and tight-binding
Hamiltonians with finite interaction range. Their spectral analysis goes back to
classical models for electrons in crystals and magnetic fields, where Azbel
\cite{Azbel1964} and Hofstadter \cite{Hofstadter1976} already observed remarkably
complicated spectral behaviour. In later work on quasiperiodic operators, Avron,
van Mouche, and Simon \cite{AvronVanMoucheSimon1990} established continuity
results for the almost Mathieu spectrum. Closely related questions for aperiodic
operators were studied by S\"ut\H{o} \cite{Suto1987} and by Bellissard, Iochum,
and Testard \cite{BellissardIochumTestard1991}. The same circle of ideas also
includes the Ten Martini Problem, resolved by Avila and Jitomirskaya
\cite{AvilaJitomirskaya2009}. In the present paper, the central objects of interest are the spectrum
$\sigma(H)$ and the $\varepsilon$--pseudospectrum
$\sigma_\varepsilon(H)$. 

Since the spectrum and pseudospectrum of the full operator are generally not
directly accessible by finite computation, one seeks to recover them from
finite-dimensional restrictions. For non-normal operators, pseudospectra provide
a substantially more stable computational target than the spectrum; see, e.g.,
\cite{TrefethenEmbree2005}. This raises the question of how much spectral
information about an infinite operator can be recovered from suitable finite
sections.

A classical approach is the \emph{finite section method}, in which an
infinite matrix is truncated to finite square submatrices. While such
square truncations work well in certain translation--invariant settings,
they often fail to provide reliable error control for more general
operators, especially in the presence of non--normality or irregular
geometry.

A more robust strategy is to work with \emph{local finite sections} that retain
all interactions emanating from a finite region. Concretely, for $L>0$ and
$x\in X$ one considers a rectangular restriction
\[
  H_{L,x}:\ell^2(B_L(x)\cap\Gamma)\longrightarrow \ell^2(B_{L+m}(x)\cap\Gamma),
\]
where $m$ is the interaction range of $H$.
The relevant spectral information is encoded by the lower norms
$\nu$ or, equivalently in finite dimensions, by the
smallest singular values of the corresponding rectangular matrices.

This leads to the following fundamental question.

\medskip
\noindent\textbf{Core question.}
\emph{How does the global lower norm $\nu(H-\lambda)$ relate to the collection of
local lower norms $\nu(H_{L,x}-\lambda)$, and can this relation be quantified
uniformly in terms of the window size $L$?}

\medskip

\noindent\textbf{Previous work.}
Two complementary lines of research motivate the present work.
On the one hand, commutator--based localisation methods on $\R^n$ use smooth or
Lipschitz cutoffs to control boundary errors for differential and difference
operators \cite{HegeMoscolariTeufel2026}. Closely related ideas also appear in \cite{Hege2024Thesis,HegeMoscolariTeufel2022}, together with applications to
spectral gaps in quasicrystals and systems of finite local complexity.
On the other hand, band--operator techniques on countable Abelian groups exploit
translation invariance to relate infinite operators to finite sections \cite{LindnerSeifertChandlerWildePreprint}. Previously, similar questions were studied
for $G=\mathbb{Z}$ in
\cite{ChandlerWildeChonchaiyaLindner2023,ChandlerWildeChonchaiyaLindner2024}, with a finite-matrix follow-up in \cite{ChandlerWildeLindner2025}.

A related perspective is provided by Beckus and Takase \cite{BeckusTakase2025},
who study dynamically-defined normal operator families on amenable groups and
derive quantitative Hausdorff continuity estimates for spectra, building in part on
earlier work on spectral continuity for aperiodic quantum systems and aperiodic
Hamiltonians \cite{BeckusBellissardDeNittis2018,BeckusBellissardCornean2019}. In particular,
their framework also applies to non-Abelian group actions, such as the discrete
Heisenberg group, and has recently been applied to spectral approximation for
aperiodic substitution systems; see \cite{BandBeckusPogorzelskiTenenbaum2025}.
All these approaches rely, implicitly or explicitly, on geometric information of the underlying structure. 
From a computational perspective, complementary approaches based on the
Solvability Complexity Index (SCI) hierarchy have been developed in
\cite{Hansen2011,BenArtziColbrookHansenNevanlinnaSeidel2020,
ColbrookHansen2023, ColbrookEmbreeFillman2026}. These works classify spectral problems according to
the number and type of limiting processes required for their solution and
provide both constructive algorithms and corresponding impossibility
results.

\medskip

\noindent\textbf{Contribution.}
The purpose of this paper is to isolate the \emph{geometric core} behind these
methods and to formulate localisation estimates that do not require $\Gamma$
itself to carry a group structure.
We work in a left--invariant doubling metric measure space $(X,d,\mu)$ and assume
only that $\Gamma\subset X$ is countable and uniformly discrete. To keep the notation
simple, we also impose relative denseness in the main body of the paper.\footnote{Relative denseness is used mainly to avoid empty local windows. Section~8 explains how the results can be adapted without this
assumption.} For the sake of clarity, we present the main arguments only in the Hilbert-space setting $\ell^2(\Gamma)$; extensions to $\ell^p(\Gamma)$ for $1\le p\le\infty$, as well as to Banach-valued spaces $\ell^2(\Gamma,E)$, are discussed separately in the final section.
For bounded band operators of finite interaction range we prove:

\begin{itemize}[leftmargin=2em]
\item a uniform commutator estimate for Lipschitz tent localisations,
\item localisation of quasi--eigenvectors with an explicit $O(1/L)$ error,
\item two--sided inclusions for window pseudospectra and Hausdorff convergence
      to the global pseudospectrum,
\item computable spectral gap tests and sampling schemes in the self--adjoint/ normal case, as well as corresponding approximation results for
      pseudospectra in the non--normal case.
\end{itemize}

The resulting error bound of order $1/L$ is essentially optimal for compactly supported
Lipschitz cutoffs.
Our framework recovers known results on $\R^n$, see \cite{HegeMoscolariTeufel2026}, and on countable Abelian groups, see \cite{LindnerSeifertChandlerWildePreprint}, as special
cases, while also covering irregular geometries arising in quasicrystal models.
In particular, it applies to operators on discrete nilpotent groups, such as the
discrete Heisenberg group, which fall outside the scope of purely Abelian or
periodic methods.

\section{Notation and analytic tools}
\label{sec:tools}

This section collects the operator--theoretic quantities used throughout the paper
to measure \emph{near--singularity} of operators and to describe (pseudo)spectral
behaviour in terms of lower bounds.  The key point is that these notions are
defined by optimisation over vectors and therefore admit natural \emph{local}
variants on finite windows, which later connect the global operator to its
finite sections.

\subsection{Lower norm}
\label{sec:lower-norms}

Let $X$ be a complex Banach space and $A\in\mathcal{B}(X)$.
The \emph{lower norm} of $A$ is
\begin{equation}\label{eq:lower-norm}
  \nu(A)
  := \inf\{\|A\psi\| : \psi\in X,\ \|\psi\|=1\}
  = \inf_{\psi\neq 0}\frac{\|A\psi\|}{\|\psi\|}.
\end{equation}
We emphasise that $\nu(\cdot)$ is not a norm (it does not satisfy a triangle
inequality), but it behaves like an ``inverse norm'' in many basic estimates;
see, e.g., \cite{Lindner2006}.

\medskip

\noindent
\textbf{Local lower norm.}
In our applications, we work on $\mathcal H=\ell^2(\Gamma)$ for an index set $\Gamma$, and
localisation is naturally expressed by support restrictions. For a nonempty set
$T\subset\Gamma$ we define the \emph{local lower norm} by
\begin{equation}\label{eq:local-lower-norm}
  \nu_T(A)
  := \inf\{\|A\psi\| : \psi\in \ell^2(\Gamma),\ \|\psi\|=1,\ \supp\psi\subset T\}.
\end{equation}
If $T\subset U\subset\Gamma$, then
\begin{equation}\label{eq:local-lower-norm-monotone}
  \nu_T(A)\;\ge\;\nu_U(A)\;\ge\;\nu(A),
\end{equation}
since the infimum is taken over a smaller class of vectors.\footnote{More conceptually, on an arbitrary Banach space $X$ one may define restricted lower norms by subspaces: for a closed subspace $Y\subset X$ set $\nu_Y(A):=\inf\{\|Ax\|:x\in Y,\ \|x\|=1\}$. In this formulation, monotonicity ($Y\subset Z\Rightarrow \nu_Y(A)\ge \nu_Z(A)$) and the perturbation bound $|\nu_Y(A)-\nu_Y(B)|\le\|A-B\|$ hold in full generality.}

\medskip

We will repeatedly use the following standard properties.  First, the (local) lower norm
is \emph{Lipschitz continuous} with respect to the operator norm: for
$A,B\in\mathcal B(\mathcal H)$,
\[
  |\nu(A)-\nu(B)|\le \|A-B\|,
  \qquad
  |\nu_T(A)-\nu_T(B)|\le \|A-B\|\quad(T\subset\Gamma).
\]
Second, $\nu(A)>0$ is equivalent to $A$ being bounded from below, i.e.\ there
exists $c>0$ such that $\|A\psi\|\ge c\,\|\psi\|$ for all $\psi\in\mathcal H$
(equivalently, $A$ is injective with closed range).  To obtain surjectivity (and hence bounded invertibility) of $A$, one also has
to control the adjoint: $A$ is surjective if and only if $A^*$ is injective
with closed range, equivalently bounded from below, i.e.\ $\nu(A^*)>0$.
It is therefore convenient to combine both quantities into the \emph{combined lower norm} 
\[ 
\nu_{comb}(A):=\min\{\nu(A),\,\nu(A^*)\}, 
\] 
where $A^*$ denotes the Hilbert-space adjoint.\footnote{More generally, one may take
the Banach adjoint; in this paper we work on $\ell^2$--spaces.}
With the convention $\|A^{-1}\|=\infty$
(and hence $\|A^{-1}\|^{-1}=0$) whenever $A$ is not invertible, one has the identity
(see \cite[Thm.~1.5.5]{Lindner2006} and \cite{Davies2007})
\begin{equation}
\label{eq: lower norm equals norm inverse}
  \|A^{-1}\|^{-1}=\nu_{comb}(A).
\end{equation}
(Our notation differs from parts of the literature where the combined lower norm is denoted by $\mu$; we reserve $\mu$ for measures.)

\medskip

\noindent
\textbf{Quasi-modes (global and local).}
The lower norm admits a natural interpretation in terms of approximate
eigenvectors. A vector $\psi\in \ell^2(\Gamma)\setminus\{0\}$ is called a
\emph{(global) $\varepsilon$-quasi-mode} for $A$ at $\lambda\in\C$ if $\|(A-\lambda)\psi\|\;\le\;\varepsilon\,\|\psi\|$.
Equivalently, $\nu(A-\lambda)$ is the infimum over all $\varepsilon$ for which such
quasi-modes exist.

If one additionally requires $\supp\psi\subset T\subset\Gamma$, then $\psi$ is
called a \emph{local $\varepsilon$-quasi-mode} (supported in $T$). In particular,
$\nu_T(A-\lambda)$ is the infimum over all $\varepsilon$ for which there exists a
quasi-mode with support contained in $T$.

In view of \eqref{eq:local-lower-norm-monotone}, the existence of local
$\varepsilon$-quasi-modes corresponding to an $\varepsilon$-pseudoeigenvalue immediately
implies the existence of global quasi-modes with the same $\varepsilon$.
The main result of this paper establishes a converse statement: the existence
of global $\varepsilon$-quasi-modes corresponding to an $\varepsilon$-pseudoeigenvalue
implies the existence of local $(\varepsilon + C_0/L)$-quasi-modes (supported in $B_L(x)$ for some
$x\in X$) corresponding to an $(\varepsilon + C_0/L)$-pseudoeigenvalue, with
explicit control of the error.

\subsection{Spectrum and pseudospectrum}
\label{sec:pseudospectrum}

Let $A\in\mathcal{B}(X)$ on a complex Banach space $X$.
The spectrum and resolvent set are
\[
  \sigma(A) := \{\lambda\in\C : A-\lambda I \text{ is not invertible}\},
  \qquad
  \rho(A):=\C\setminus\sigma(A),
\]
and the resolvent is $R(\lambda,A):=(A-\lambda I)^{-1}$ for $\lambda\in\rho(A)$.

\medskip

\noindent
\textbf{Pseudospectra via lower norms.}
For $\varepsilon>0$ we define the \emph{open} and \emph{closed}
$\varepsilon$--pseudospectra of $A$ by
\begin{align}
  \sigma_\varepsilon(A)
  &:= \{\lambda\in\C : \nu_{comb}(A-\lambda I) < \varepsilon\},\label{eq:pseudo-open}\\
  \Sigma_\varepsilon(A)
  &:= \{\lambda\in\C : \nu_{comb}(A-\lambda I) \le \varepsilon\}.\label{eq:pseudo-closed}
\end{align}
By \eqref{eq: lower norm equals norm inverse}, these definitions are equivalent to the
classical resolvent--norm formulation (see, e.g.,
\cite{Globevnik1976,DaviesShargorodsky2015,TrefethenEmbree2005}):
\[
  \sigma_\varepsilon(A)
  = \{\lambda\in\C : \|R(\lambda,A)\| > \varepsilon^{-1}\},
  \qquad
  \Sigma_\varepsilon(A)
  = \{\lambda\in\C : \|R(\lambda,A)\| \ge \varepsilon^{-1}\},
\]
where we interpret $\|R(\lambda,A)\|=\infty$ for $\lambda\in\sigma(A)$.

\medskip

\noindent
As preimages of open, respectively closed, subsets of\/ $\R$ under the continuous map
$\lambda\mapsto \nu_{comb}(A-\lambda I)$, the sets $\sigma_\varepsilon(A)$ and
$\Sigma_\varepsilon(A)$ are open, respectively closed.

\medskip

\noindent
\textbf{Open vs.\ closed pseudospectra.}
By continuity one always has $\overline{\sigma_\varepsilon(A)}\subset
\Sigma_\varepsilon(A)$. Equality can fail in general Banach spaces if
$\nu_{comb}(A-\lambda I)$ has a level set with interior points\footnote{This means that there exist $\lambda_0\in\C$ and $r>0$ such that
$\nu_{comb}(A-\lambda I)$ is constant on the ball $B_r(\lambda_0)$, i.e.\ the level set
$\{\lambda:\nu_{comb}(A-\lambda I)=\nu_{comb}(A-\lambda_0 I)\}$ contains a nonempty open subset.} (see
\cite{Shargorodsky2008}). However, for a large class of spaces --- in
particular for Hilbert spaces --- the resolvent norm cannot be locally constant
on any open set (results of \cite{Globevnik1976,Shargorodsky2008}).\footnote{One sufficient criterion is (complex) uniform convexity of $X$ or of
its dual $X^*$: in this case the resolvent norm cannot be locally constant on
any nonempty open subset of $\rho(A)$, hence
$\overline{\sigma_\varepsilon(A)}=\Sigma_\varepsilon(A)$ for all $\varepsilon>0$;
see \cite{Globevnik1976} and \cite{Shargorodsky2008}, and also the discussion in
\cite{DaviesShargorodsky2015}. In particular, every Hilbert space belongs to
this class.}
In this ``Globevnik class'' one has
\begin{equation}\label{eq:open-closed-agree}
  \Sigma_\varepsilon(A)=\overline{\sigma_\varepsilon(A)}\qquad(\varepsilon>0).
\end{equation}
In the Hilbert space setting, \eqref{eq:open-closed-agree} is always available.
\subsection{Rectangular operators and local pseudospectra}
\label{sec:rectangular-pseudo}
Later we will consider finite sections (or local restrictions) of operators,
which naturally lead to \emph{rectangular} maps between two different
$\ell^2$--spaces. In this situation it is convenient to define pseudospectra
purely in terms of lower norms.

\begin{definition}[Rectangular $\varepsilon$--pseudospectrum]\label{def:rectangular-pseudo}
Let $T\subset T'\subset\Gamma$ be finite and let
\[
  A:\ell^2(T)\longrightarrow \ell^2(T')
\]
be bounded. Denote by $I_T:\ell^2(T)\to\ell^2(T')$ the canonical embedding
\[
  (I_T x)(t) =
  \begin{cases}
    x(t), & t\in T,\\
    0,    & t\in T'\setminus T.
  \end{cases}
\]
The \emph{rectangular open} and \emph{rectangular closed} $\varepsilon$--pseudospectra of $A$ are
\[
  \sigma_\varepsilon(A)
  := \{\lambda\in\C : \nu(A-\lambda I_T) < \varepsilon\},
  \qquad \varepsilon>0,
\]
and
\[
  \Sigma_\varepsilon(A)
  := \{\lambda\in\C : \nu(A-\lambda I_T) \le \varepsilon\},
  \qquad \varepsilon\ge 0.
\]
Since $T,T'$ are finite, $\nu(A-\lambda I_T)$ equals the minimal singular value of $A-\lambda I_T$,
i.e.\ $\nu(A-\lambda I_T)^2$ is the smallest eigenvalue of $(A-\lambda I_T)^*(A-\lambda I_T)$.
\end{definition}

\begin{remark}[Why we do not use the combined lower norm]\label{rem:rectangular-no-combined}
Let $B:=A-\lambda I_T$. If $T\subsetneq T'$, then $B:\ell^2(T)\to\ell^2(T')$ with
$\dim\ell^2(T')>\dim\ell^2(T)$. Hence $B^*:\ell^2(T')\to\ell^2(T)$ is never
injective, so $\nu(B^*)=0$ for all $\lambda$. The combined quantity
$\nu_{comb}(B)=\min\{\nu(B),\nu(B^*)\}$ therefore carries no additional
information in the rectangular setting. If $T=T'$, then $\nu(B^*)=\nu(B)$ and therefore
$\nu_{\mathrm{comb}}(B)=\nu(B)$.
\end{remark}

\subsection{Hausdorff convergence and the Globevnik property}
\label{sec:hausdorff-globevnik}

When we speak about convergence of set--valued spectral approximations, we use
Hausdorff distance on bounded subsets of\/ $\C$.  For nonempty bounded sets
$S,T\subset\C$ we set
\[
  d_H(S,T)
  := \max\Bigl\{
      \sup_{s\in S}\dist(s,T),\,
      \sup_{t\in T}\dist(t,S)
    \Bigr\},
  \qquad
  \dist(z,T):=\inf_{w\in T}|z-w|.
\]
This defines a metric on compact subsets of\/ $\C$ and, more generally, a
pseudometric on bounded subsets; in particular, on bounded sets the Hausdorff pseudometric only depends on closures,
\[
  d_H(S,T)=0 \quad\Longleftrightarrow\quad \overline{S}=\overline{T}.
\]

\medskip

With \eqref{eq:open-closed-agree} at hand,
one may freely switch between open and closed $\varepsilon$--pseudospectra when
discussing Hausdorff convergence: for fixed $\varepsilon>0$ and any bounded sets
$S_n,S\subset\C$,
\[
  S_n \xrightarrow[d_H]{} \sigma_\varepsilon(A)
  \quad\Longleftrightarrow\quad
  S_n \xrightarrow[d_H]{} \Sigma_\varepsilon(A).
\]
Moreover, on spaces satisfying the Globevnik property, the family of
$\varepsilon$--pseudospectra depends Hausdorff--continuously on $\varepsilon>0$.\footnote{For the closed sets $\Sigma_\varepsilon(A)=\{\nu_{comb}(A-\lambda I)\le\varepsilon\}$,
the monotone limit for $\varepsilon_n\downarrow\varepsilon$ is immediate from
$\bigcap_n\Sigma_{\varepsilon_n}(A)=\Sigma_\varepsilon(A)$.
The Globevnik property is only needed for the monotone increase
$\varepsilon_n\uparrow\varepsilon$, since then
$\bigcup_n\Sigma_{\varepsilon_n}(A)=\{\nu_{comb}(A-\lambda I)<\varepsilon\}$
and one must identify the closure of this union with $\Sigma_\varepsilon(A)$,
i.e.\ exclude locally constant level sets of the resolvent norm.}
In particular, if $\varepsilon_n\to\varepsilon>0$, then
\[
  d_H\bigl(\sigma_{\varepsilon_n}(A),\sigma_\varepsilon(A)\bigr)\to 0,
  \qquad
  d_H\bigl(\Sigma_{\varepsilon_n}(A),\Sigma_\varepsilon(A)\bigr)\to 0.
\]

Moreover, for the closed pseudospectra, the monotone convergence from above extends to $\varepsilon=0$ without requiring the Globevnik property, since
\[
    \bigcap_{\eta>0}\Sigma_\eta(A)
    = \Sigma_0(A)
    = \sigma(A).
\]

\section{Geometric and operator setting}

\subsection{Left-invariant metric--measure space}

Our localisation argument requires three ingredients on the ambient space $X$:
\emph{(i)} a notion of balls (hence a metric), \emph{(ii)} a way to average over
centres $x$ (hence a measure), and \emph{(iii)} compatibility with translations
so that the resulting estimates do not depend on the chosen centre.

\begin{assumption}[Standing left-invariant framework]\label{ass:standing}
Throughout, $(X,d,\mu)$ is a metric measure space such that
\begin{itemize}[leftmargin=2em]
\item $X$ is a group with identity $e$,
\item $d$ is a \emph{left-invariant} metric on $X$, i.e.\ $d(gx,gy)=d(x,y)$ for all $g,x,y\in X$, and $X$ is endowed with the topology induced by $d$
\item $\mu$ is a non-zero \emph{left-invariant} Radon measure on $X$ with respect to this topology, i.e. $\mu(gE)=\mu(E)$ for all Borel sets $E\subset X$ and $g\in X$,
\item the ball volume function
      \[
        V(r):=\mu(B_r(e)), \qquad B_r(x):=\{y\in X:d(x,y)\le r\},
      \]
      satisfies 
      \[
        0<V(r)<\infty \qquad\text{for all }r>0,
      \]
\item $V$ is \emph{doubling}, i.e.\ there exists $C_{\mathrm{dbl}}\ge1$ such that
      \[
        V(2r)\le C_{\mathrm{dbl}}V(r) \qquad\text{for all }r>0.
      \]
\end{itemize}
\end{assumption}

\medskip

This framework covers many natural examples. If $X$ is a locally
compact Hausdorff group, then Haar's theorem (see, e.g.,
\cite{folland95}) provides a non-zero left-invariant Radon measure
$\mu$. If, in addition, $X$ is second-countable, Struble's theorem
\cite{Struble1974} yields a proper left-invariant metric $d$ inducing
the topology. For such a choice of $d$ and $\mu$, closed metric balls
are compact and hence have finite measure, while nonempty open balls
have positive measure. Thus $0<V(r)<\infty$ for every $r>0$.
Although the doubling condition is a genuine geometric restriction,
the framework includes many standard examples, such as finitely
generated groups of polynomial growth and non-Abelian groups
(e.g.\ the discrete Heisenberg group).

\medskip

The standing assumptions also ensure the existence of countable, uniformly discrete and relatively dense subsets $\Gamma$ of $X$. Indeed, for any $r>0$, a maximal $r$-separated subset is relatively dense by maximality, while the finite-volume assumption on metric balls implies local finiteness and hence countability; see the packing argument below.

\subsection{Band operators and finite sections}
\label{sec:band-and-windows}

We now specify the operator class for which the lower-norm and pseudospectral
tools from the previous section will be applied.  The guiding principle is
\emph{finite-range interactions}: the action of the operator should be
determined by local couplings within a fixed metric radius.  This is the
natural setting for discrete Schr\"odinger-type models on uniformly discrete
sets, and it is precisely the locality that later allows us to compare global
and local lower norms.

\paragraph{Band operators (finite-interaction-range).}
Let $\Gamma\subset X$ be a countable, uniformly discrete, relatively dense\footnote{A subset $\Gamma\subset X$ is called uniformly discrete if there exists $r>0$ such that $d(p,q)\ge r$ for all distinct $p,q\in\Gamma$. The supremum of all such $r$ is called the separation radius of $\Gamma$. It is called relatively dense if there exists $R>0$ such that $B_R(x)\cap\Gamma\neq\emptyset$ for all $x\in X$. The infimum of all such $R$ is called the covering radius of $\Gamma$.} subset, and consider the
Hilbert space $\ell^2(\Gamma)$. A bounded operator
$H:\ell^2(\Gamma)\to \ell^2(\Gamma)$ is called a \emph{band operator} (or
\emph{finite-interaction-range operator}) if there exist constants $m>0$ and $M>0$ such that
its matrix entries
\[
  H_{pq} := \langle \delta_p,\,H\delta_q\rangle, \qquad p,q\in\Gamma,
\]
satisfy the uniform bound $|H_{pq}|\le M$ and the finite-interaction-range condition
\[
  d(p,q) > m \quad \Longrightarrow \quad H_{pq}=0.
\]
Equivalently, $H\delta_q$ is supported in $B_m(q)\cap\Gamma$ for every $q\in\Gamma$.
Thus $m$ controls the maximal interaction radius (or bandwidth), while $M$ provides a uniform
bound on the coupling strengths.\footnote{We work with closed balls in order to allow for a minimal choice of the interaction range $m$. Indeed, if $p\in B_L(x)$ and $q\in\Gamma$ satisfies $d(p,q)\le m$, then $q$ may lie on the boundary of $B_{L+m}(x)$, so that $B_{L+m}(x)$ is the smallest set that still captures all interactions.}

\paragraph{Interaction degree.}
Besides the interaction-range $m$, we will repeatedly use a uniform bound on the number of
nontrivial off-diagonal couplings per site.  We therefore set
\[
  \gamma
  := \sup_{q\in \Gamma}\#\bigl\{p\in \Gamma\setminus\{q\}:\ H_{pq}\neq 0 \;\; \text{or}\;\; H_{qp} \neq 0\bigr\}.
\]
In our standing metric--measure setting, $\gamma$ is automatically finite.
Indeed, if $H$ has interaction-range $m$, then $H_{pq}\neq 0$ implies
$d(p,q)\leq m$, hence
\[
  \gamma
  \;\le\;
  \sup_{q\in\Gamma}\#\bigl((B_m(q)\setminus\{q\})\cap\Gamma\bigr).
\]
Since $\Gamma$ is uniformly discrete, there exists $0<r<r_\Gamma$ such that the balls
$B_{r/2}(p)$, $p\in\Gamma$, are pairwise disjoint. For fixed $q\in\Gamma$,
all balls $B_{r/2}(p)$ with $p\in B_m(q)\cap\Gamma$ are contained in
$B_{m+r/2}(q)$, and by left-invariance of $\mu$ we obtain the packing bound
\[
  \#\bigl(B_m(q)\cap\Gamma\bigr)\,\mu\bigl(B_{r/2}(e)\bigr)
  \;\le\; \mu\bigl(B_{m+r/2}(q)\bigr)
  \;=\; \mu\bigl(B_{m+r/2}(e)\bigr).
\]
Consequently,
\[
  \gamma \;\le\; \frac{\mu\bigl(B_{m+r/2}(e)\bigr)}{\mu\bigl(B_{r/2}(e)\bigr)} - 1 \;<\;\infty,
\]
so the interaction degree is controlled purely by the geometry of
$(X,d,\mu)$, the separation radius $r_\Gamma$ of $\Gamma$, and the interaction radius $m$. The same argument shows that $B_L(x)\cap\Gamma$ is finite for every $L>0$ and $x\in X$. In particular, all spaces $\ell^2(B_L(x)\cap\Gamma)$ occurring below are finite-dimensional.

For later reference, and to avoid repeating the same hypotheses in each statement, we also fix the operator setting used throughout the
subsequent results.

\begin{assumption}[Standing operator setting]\label{ass:operators}
From now on, let $\Gamma\subset X$ be
countable, uniformly discrete and relatively dense. Let
$H$ be a band operator on $\ell^2(\Gamma)$ with interaction range
$m>0$ and matrix bound $M>0$, in the sense introduced above. We denote by
$\gamma$ the corresponding interaction degree.
\end{assumption}

\paragraph{Finite sections as rectangular operators.}
Our goal is to approximate the \emph{global lower norm} \(\nu(H-\lambda)\),
but in computations we can only test vectors with \emph{finite support}.
For band operators this is particularly natural: if \(\psi\) is supported in a
finite set, then \(H\psi\) is again supported in a finite neighbourhood, and
the possible spread of the support is controlled explicitly by the interaction
range. This motivates finite-dimensional (yet faithful) truncations of \(H\).

For \(L>0\) and \(x\in X\), let
\[
  P_{L,x}:\ell^2(\Gamma)\to \ell^2\bigl(B_L(x)\cap\Gamma\bigr)
\]
denote the orthogonal projection (restriction) onto \(\ell^2(B_L(x)\cap\Gamma)\),
and let \(P_{L,x}^*\) be the canonical embedding back into \(\ell^2(\Gamma)\)
(extension by zero).\footnote{Here relative denseness ensures that for sufficiently large \(L\), the sets \(B_L(x)\cap\Gamma\) are nonempty for all \(x\in X\), so that the finite sections are defined on nontrivial spaces. One can avoid this assumption by modifying the construction; see Section~8.}

Because \(H\) has interaction range \(m\), applying \(H\) to a vector supported in
\(B_L(x)\) may create nonzero components in the enlarged ball \(B_{L+m}(x)\), but
not beyond. In other words, the \emph{input} lives on \(B_L(x)\), while the
\emph{output} is contained in \(B_{L+m}(x)\). This leads to the following
\emph{(rectangular)} finite section:
\[
  H_{L,x}
  := P_{L+m,x}\,H\,P_{L,x}^*
  \;:\;
  \ell^2\bigl(B_L(x)\cap\Gamma\bigr)
  \longrightarrow
  \ell^2\bigl(B_{L+m}(x)\cap\Gamma\bigr).
\]
Equivalently, \(H_{L,x}\) is obtained from \(H\) by keeping exactly those matrix
entries \(H_{pq}\) for which \(q\in B_L(x)\cap\Gamma\) and
\(p\in B_{L+m}(x)\cap\Gamma\), and discarding all other rows/columns.

\paragraph{Local lower norms for finite sections.}
\label{par: local lower norms}
For $\lambda\in\C$ we write $H-\lambda := H-\lambda I$ and define the corresponding
rectangular shift on the window by
\[
  (H_{L,x}-\lambda)
  := P_{L+m,x}\,(H-\lambda)\,P_{L,x}^*
  \;=\; H_{L,x}-\lambda I_{L,x},
\]
where $I_{L,x}:\ell^2(B_L(x)\cap\Gamma)\to \ell^2(B_{L+m}(x)\cap\Gamma)$ denotes the
canonical embedding (extension by zero).

Since the lower norm is defined as infimum over all vectors $\psi \in \ell^2(\Gamma)$,
restricting to vectors supported in $B_L(x)\cap\Gamma$ can only increase it.
In particular, for every $x\in X$ and $\lambda\in\C$,
\begin{equation}\label{eq:local-lower-bound}
  \nu(H-\lambda)\;\le\;\nu\bigl(H_{L,x}-\lambda\bigr).
\end{equation}
Applying the same observation to $(H-\lambda)^* =H^*-\overline{\lambda}$ yields
\[
  \nu_{comb}(H-\lambda)
  :=\min\{\nu(H-\lambda),\,\nu(H^*-\overline{\lambda})\}
  \;\le\;
  \min\Bigl\{\nu\bigl(H_{L,x}-\lambda\bigr),\,
             \nu\bigl((H^*)_{L,x}-\overline{\lambda}\bigr)\Bigr\}.
\]
Taking the infimum over all window centres $x\in X$ preserves this inequality, hence
for every $\lambda\in\C$,
\[
  \nu_{comb}(H-\lambda)
  \;\le\;
  \inf_{x\in X}\min\Bigl\{\nu\bigl(H_{L,x}-\lambda\bigr),\,
                         \nu\bigl((H^*)_{L,x}-\overline{\lambda}\bigr)\Bigr\}. \footnote{Notice that $(H^*)_{L,x} \neq (H_{L,x})^*$. Hence the right-hand side cannot be rewritten directly in terms of the combined lower norm $\nu_{comb}(H-\lambda)$.}
\]
The main contribution of this paper is an explicit \emph{upper bound} on the local
lower norms in terms of the (global) lower norm: there exist constants $C_0, L_0>0$,
independent of $L$, such that
\[
  \inf_{x\in X}\nu\bigl(H_{L,x}-\lambda\bigr)
  \;\le\;
  \nu(H-\lambda)+\frac{C_0}{L},
  \qquad \lambda\in\C,\ L>L_0.
\]
Again, applying the same estimate to $H^*-\overline{\lambda}$ and taking the minimum yields
immediately the corresponding bound for the combined lower norm,
\[
  \inf_{x\in X}\min\Bigl\{\nu\bigl(H_{L,x}-\lambda\bigr),\,
                         \nu\bigl((H^*)_{L,x}-\overline{\lambda}\bigr)\Bigr\}
  \;\le\;
  \nu_{comb}(H-\lambda)+\frac{C_0}{L}. 
\]
In particular, the infimum of local lower norms approximates the global lower norm
from above with an explicit error of order $O(1/L)$.

\begin{remark}[From $\inf$ to $\min$]\label{rem:inf-min-finite}
For fixed $L>0$ and $\lambda\in\C$ one may interchange the order in which the
infimum over centres and the minimum over the two local lower norms are taken:
\begin{equation}\label{eq:inf-min-swap}
  \inf_{x\in X}\min\Bigl\{\nu(H_{L,x}-\lambda),\,\nu\bigl((H^*)_{L,x}-\overline{\lambda}\bigr)\Bigr\}
  \;=\;
  \min\Bigl\{\inf_{x\in X}\nu(H_{L,x}-\lambda),\,\inf_{x\in X}\nu\bigl((H^*)_{L,x}-\overline{\lambda}\bigr)\Bigr\}.
\end{equation}
\end{remark}

\paragraph{Intermezzo: finite local complexity.}
\label{intermezzo: flc}
Moreover, the infima in \eqref{eq:inf-min-swap} reduce to minima as soon as, for the given scale $L$, only finitely many finite sections $H_{L,x}$ (and $(H^*)_{L,x}$) occur, up to the relevant notion of equivalence. To make this precise, fix $L>0$ and, for each $x\in X$, consider the local pair
\[
  \bigl(B_{L+m}(x)\cap\Gamma,\; H_{L,x}\bigr).
\]
Two such pairs
\[
  \bigl(B_{L+m}(x)\cap\Gamma,\; H_{L,x}\bigr)
  \qquad\text{and}\qquad
  \bigl(B_{L+m}(y)\cap\Gamma,\; H_{L,y}\bigr)
\]
are called equivalent if there exists an isometry $\tau:X\to X$ from a prescribed class of isometries, for example, compositions of translations and rotations, such that
\[
  \tau\bigl(B_L(x)\cap\Gamma\bigr)=B_L(y)\cap\Gamma
  \qquad\text{and}\qquad
  \tau\bigl(B_{L+m}(x)\cap\Gamma\bigr)=B_{L+m}(y)\cap\Gamma,
\]
and such that the induced unitary operators
\[
  U_\tau:\ell^2(B_L(x)\cap\Gamma)\to \ell^2(B_L(y)\cap\Gamma),
  \qquad
  V_\tau:\ell^2(B_{L+m}(x)\cap\Gamma)\to \ell^2(B_{L+m}(y)\cap\Gamma)
\]
satisfy
\[
  H_{L,y}\,U_\tau=V_\tau\,H_{L,x}.
\]
If, for fixed $L$, only finitely many such equivalence classes occur, then the relevant infima reduce to minima over a finite list of representatives. If this property holds for every $L>0$, we say that $(\Gamma,H)$ has
\emph{finite local complexity} with respect to the chosen class of
isometries. This may be viewed as an operator-decorated version of the
usual notion of finite local complexity for Delone sets and aperiodic
structures; for background on the latter, see
\cite{BaakeGrimm2013}, and for its role in
spectral approximation from local patches, see
\cite{Hege2024Thesis,HegeMoscolariTeufel2026}.

\begin{remark}[Standing assumptions]
\label{rem:standing-assumptions}
From this point on, we work throughout under Assumption~\ref{ass:standing}
on $(X,d,\mu)$, with a countable, uniformly discrete, relatively dense set $\Gamma\subset X$,
and with finite-interaction-range operators $H:\ell^2(\Gamma)\to\ell^2(\Gamma)$ with constants
$m$, $M$, and $\gamma$ as described above.
\end{remark}

\subsection{Finite sections: two guiding examples}
\label{sec:finite-sections-examples}

The term \emph{finite section} comes from the classical situation of band
operators on $\ell^2(\mathbb{Z})$, where one literally cuts out a finite block (a
\emph{section}) of the infinite matrix.  In our setting the same idea persists:
one restricts the support to a finite set, while keeping precisely
those rows needed to record all interactions of a finite--range operator.
We record two concrete pictures (one matrix-based, one graph-based) that
illustrate this restriction procedure and match the localised operators used
later.

\begin{example}[Truncating an infinite band matrix on $\ell^2(\mathbb{Z})$]
Let $X=\mathbb{Z}$ with the usual metric, fix $x=0$, and consider a band operator
$H:\ell^2(\mathbb{Z})\to\ell^2(\mathbb{Z})$ of interaction range $m=1$.  In the standard basis
$(\delta_n)_{n\in\mathbb{Z}}$ the matrix of $H$ is tridiagonal:
\[
H \;=\;
\begin{pmatrix}
\ddots & \ddots & \ddots &        &        &        \\
       & a_{-2} & b_{-1} & c_{0}  &        &        \\
       &        & a_{-1} & b_{0}   & c_{1}  &        \\
       &        &        & a_{0}   & b_{1}  & c_{2}  \\
       &        &        &         & \ddots & \ddots \\
\end{pmatrix}
\]

\medskip
For $L\in\mathbb{N}$, the window is $B_{L}(0)=\{-L,\dots,L\}$ and the enlarged window is
$B_{L+m}(0)=\{-(L+m),\dots,L+m\}$.  The corresponding rectangular finite section
is
\[
  H_{L,0}
  \;:=\;
  P_{L+m,0}\,H\,P_{L,0}^*
  \;:\;
  \ell^2\!\bigl(B_L(0)\bigr)\longrightarrow \ell^2\!\bigl(B_{L+m}(0)\bigr).
\]
Equivalently, $H_{L,0}$ consists of the columns of $H$ indexed by
$q\in B_L(0)$, restricted to the rows indexed by $p\in B_{L+m}(0)$; i.e.\ it
keeps precisely those matrix entries $H_{pq}$ with
$q\in\{-L,\dots,L\}$ and $p\in\{-(L+m),\dots,L+m\}$.

\medskip
For a concrete picture, take $L=2$ (and still $m=1$). Then
$B_{2}(0)\cap\Gamma=\{-2,-1,0,1,2\}$ and $B_{3}(0)\cap\Gamma=\{-3,-2,-1,0,1,2,3\}$, so
$H_{2,0}:\ell^2(\{-2,\dots,2\})\to \ell^2(\{-3,\dots,3\})$ is the
$(7\times 5)$ matrix
\[
H_{2,0}=
\left(
\begin{array}{ccccc}
 c_{-2} & 0     & 0     & 0     & 0     \\
 b_{-2} & c_{-1}& 0     & 0     & 0     \\
 a_{-2} & b_{-1}& c_{0}& 0     & 0     \\
 0      & a_{-1}& b_{0} & c_{1} & 0     \\
 0      & 0     & a_{0} & b_{1} & c_{2} \\
 0      & 0     & 0     & a_{1} & b_{2} \\
 0      & 0     & 0     & 0     & a_{2}
\end{array}
\right),
\]
where the first and last row represent boundary output into
$B_{3}(0)\setminus B_{2}(0)$.
\end{example}

\begin{example}[Truncating a graph Schr\"odinger operator]
Let \(X\coloneqq \R^2\) be equipped with the Euclidean metric, and let \(\Gamma\subset X\) be the lattice generated by $f\coloneqq (1,0)$ and $g\coloneqq (1,1),$
that is,
\[
\Gamma=\{\, i f + j g : i,j\in\Z \,\}
      =\{(i+j,j): i,j\in\Z\}.
\]
We connect $p,q\in\Gamma$ by an undirected edge if $p-q \in \{f, -f, g, -g\}$,
so each vertex has degree $4$. On $\ell^2(\Gamma)$ we consider the discrete
Schr\"odinger operator
\[
  (H\psi)(p)=(\Delta\psi)(p)+V(p)\psi(p),
  \qquad
  (\Delta\psi)(p):=\sum_{q\sim p}\bigl(\psi(q)-\psi(p)\bigr),
\]
where $V:\Gamma\to\R$ is a bounded potential and $q\sim p$ denotes adjacency.

Fix a centre $x\in X$ and a radius $L>0$. The (rectangular) finite section
\[
  H_{L,x}=P_{L+m,x}\,H\,P_{L,x}^*
  :\ell^2(B_L(x)\cap\Gamma)\to\ell^2(B_{L+m}(x)\cap\Gamma)
\]
keeps all interactions emanating from the inner window $B_L(x)\cap\Gamma$ while
allowing the output to spill into the one-step enlarged window
$B_{L+m}(x)\cap\Gamma$ (here the interaction range is $m=\sqrt{2}$).

\medskip

\noindent
A schematic picture is shown below: black vertices form the input window
$B_L(x)\cap\Gamma$, while grey vertices represent the boundary layer
$(B_{L+\sqrt{2}}(x)\setminus B_L(x))\cap\Gamma$ needed to capture the output of $H$.

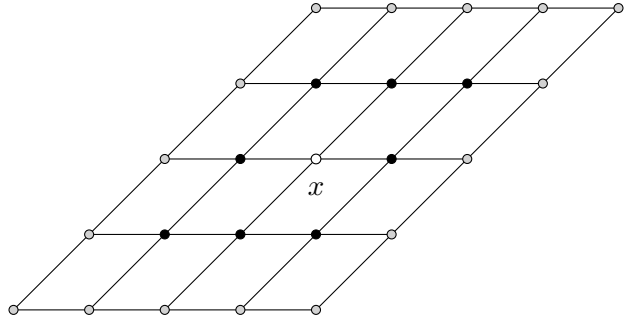
\begin{center}
\begin{tikzpicture}[
  scale=1.0,
  vertex/.style={circle,draw,inner sep=1.2pt},
  vin/.style={vertex,fill=black},
  vout/.style={vertex,fill=gray!35},
  every edge/.style={draw,gray!70,line width=0.6pt},
]
\foreach \i in {-2,-1,0,1,2}{
  \foreach \j in {-2,-1,0,1,2}{
    \coordinate (p\i\j) at ({\i+\j},{\j});
  }
}

\foreach \i in {-2,-1,0,1}{
  \foreach \j in {-2,-1,0,1,2}{
    \draw (p\i\j) -- (p\the\numexpr\i+1\relax\j);
  }
}
\foreach \i in {-2,-1,0,1,2}{
  \foreach \j in {-2,-1,0,1}{
    \draw (p\i\j) -- (p\i\the\numexpr\j+1\relax);
  }
}

\foreach \i in {-2,-1,0,1,2}{
  \foreach \j in {-2,-1,0,1,2}{
    \pgfmathtruncatemacro{\ai}{abs(\i)}
    \pgfmathtruncatemacro{\aj}{abs(\j)}
    \ifnum\ai<2
      \ifnum\aj<2
        \node[vin] at (p\i\j) {};
      \else
        \node[vout] at (p\i\j) {};
      \fi
    \else
      \node[vout] at (p\i\j) {};
    \fi
  }
}

\node[vertex,fill=white,draw=black,inner sep=1.3pt] (x) at (p00) {};
\node[below=3pt of x] {$x$};

\end{tikzpicture}
\end{center}

\noindent
\emph{Interpretation.}
If $\psi$ is supported on the black vertices (inner window), then $(H\psi)(p)$
can be nonzero on the black and also on the grey vertices, because the Laplacian couples to
nearest neighbours. The rectangular restriction $H_{L,x}=P_{L+\sqrt{2},x}HP_{L,x}^*$
captures exactly this effect: it restricts the input to the inner window while
keeping the minimal enlargement needed for the output.
\end{example}

\section{Tent functions and localisation}
\label{sec:tent}

To relate the global lower norm of an operator to lower norms on finite windows,
we need a localisation procedure on $\Gamma$ that interacts well with the
geometry and with the finite interaction range of $H$. A natural first idea is
to use \emph{hard cut-offs}, i.e.\ characteristic functions of metric balls.
However, sharp truncations typically create large boundary effects: commuting a
band operator with a hard cut-off produces a boundary term that does not admit a
useful decay in the window size.

Following the strategy used successfully in Euclidean settings (see, e.g.,
\cite{HegeMoscolariTeufel2026}), we therefore use \emph{Lipschitz} cut-offs, namely
tent functions. The key advantage is that tent profiles are uniformly Lipschitz,
so differences $w_{L,x}(p)-w_{L,x}(q)$ can be controlled by $d(p,q)$, which is
bounded by the interaction radius whenever $H_{pq}\neq 0$. This is the basic
mechanism that later yields commutator bounds of order $1/L$.

\medskip

Let $(X,d,\mu)$ be as in the standing framework (Assumption~\ref{ass:standing}).
Since \(X\) carries the topology induced by \(d\), distance functions are continuous and hence Borel measurable. This ensures that the tent functions
introduced below are measurable and can be integrated with respect to $\mu$.

\begin{definition}[Tent functions and localisation operators]
\label{def:tent}
For $L>0$ and $x\in X$ we define the \emph{tent function}
\[
  w_{L,x}(y)
  \;:=\;
  \max\Bigl\{0,\; 1-\frac{d(x,y)}{L}\Bigr\},
  \qquad y\in X.
\]
Restricting $w_{L,x}$ to $\Gamma\subset X$, we define the associated
multiplication operator $W_{L,x}$ on $\ell^2(\Gamma)$ by
\[
  (W_{L,x}\psi)(p) \;:=\; w_{L,x}(p)\,\psi(p),
  \qquad p\in\Gamma.
\]
\end{definition}

\noindent\textit{Measurability.}
For fixed $x\in X$, the map $y\mapsto d(x,y)$ is continuous (hence Borel
measurable), and $\rho\mapsto \max\{0,1-\rho/L\}$ is continuous on $\R$.
Therefore $w_{L,x}$ is a Borel function on $X$. In particular,
$\supp(w_{L,x})\subset B_L(x)$ is a Borel set, and integrals of $w_{L,x}$ and
$|w_{L,x}|^2$ with respect to $\mu$ are well-defined. We summarize:

\begin{lemma}[Measurability, support, and Lipschitz bound]
\label{lem:tent-regularity}
Let $(X,d,\mu)$ be a metric measure space such as in Assumption \ref{ass:standing}, $L>0$ and $x\in X$. It holds that:
\begin{enumerate}[leftmargin=2em]
\item[(i)] $w_{L,x}$ is Borel measurable on $X$ (hence $\mu$-measurable);
\item[(ii)] $\supp(w_{L,x})\subset B_L(x)$;
\item[(iii)] $w_{L,x}$ is $\tfrac1L$--Lipschitz, i.e.
\[
  |w_{L,x}(y)-w_{L,x}(z)|
  \;\le\;\frac{1}{L}\,d(y,z)
  \qquad (y,z\in X).
\]
\end{enumerate}
\end{lemma}

\medskip

The doubling property of the volume function $V(r)=\mu(B_r(e))$ will later be
used to relate $\|w_{L,e}\|_{L^2(X)}^2$ to $V(L+m)$.

\section{A commutator bound from doubling geometry}

The commutator $[W_{L,x},H] \coloneqq W_{L,x}H-HW_{L,x}$ captures the localisation error when pulling $H$ through the cutoff. Its $L^2$ norm is controlled by the Lipschitz constant of $w_{L,x}$ and by the geometry through the volume growth.

The following estimate shows that this localisation error is small in an averaged
sense and decays at rate $O(1/L)$, with constants determined by the geometry and
the interaction range.

\begin{theorem}[Commutator bound with tent weights]\label{thm:commutator}
Let $(X,d,\mu)$, $\Gamma \subset X$ and $H$ satisfy the settings of Assumptions~\ref{ass:standing}
and~\ref{ass:operators}.
For $L>0$ and $x\in X$ define $W_{L,x}$ as above. Then for all $\psi\in\ellTwo(\Gamma)$,
\begin{equation}\label{eq:commutator}
  \int_X \norm{[W_{L,x},H]\psi}_{\ellTwo(\Gamma)}^2\,d\mu(x)
  \;\le\;
  \frac{\gamma^2\,C_{\mathrm{geom}}\,m^2M^2}{L^2}\,
  \norm{\psi}_{\ellTwo(\Gamma)}^2\,\norm{w_{L,e}}_{L^2(X)}^2,
\end{equation}
where
\[
  C_{\mathrm{geom}}\coloneqq 4\,C_{\mathrm{dbl}}^{\,k},
  \qquad
  k=\Bigl\lceil\log_2\bigl(2(1+m/L)\bigr)\Bigr\rceil.
\]
In particular, for $L\ge m$, one may bound $C_{\mathrm{geom}}\le 4C_{\mathrm{dbl}}^2$ uniformly in $L$.
\end{theorem}

\begin{proof}
Let $\psi \in \ell^2(\Gamma)$ be fixed. The squared norm of the commutator is
\begin{align*}
  R(x)
  \,\coloneqq &\ \bigl\|[W_{L,x}, H] \psi\bigr\|_{\ell^2(\Gamma)}^2
       =  \sum_{p \in \Gamma} \left| \sum_{q \in \Gamma}
        \bigl(w_{L,x}(p) - w_{L,x}(q)\bigr) H_{pq} \psi(q)\right|^2.
\end{align*}
Since $H$ has finite interaction range $m$, the inner sum runs only over $q$ with $d(p,q) \le m$. Moreover, the term vanishes unless one of $p,q$ lies in $B_L(x)$. Since $d(p,q)\le m$, this implies that non-zero terms only occur if $q \in B_{L+m}(x)$.
Using the Lipschitz property $|w_{L,x}(p)-w_{L,x}(q)| \le d(p,q)/L \le m/L$ and the bound $|H_{pq}|\le M$, we estimate
\[
  R(x) \;\le\; \sum_{p \in \Gamma} \left( \sum_{q \sim p} \frac{mM}{L} |\psi(q)| \right)^2 \;=\; \frac{m^2M^2}{L^2}\sum_{p \in \Gamma} \left( \sum_{q \sim p} |\psi(q)| \right)^2,
\]
where the inner sum is over neighbours $q$ of $p$ (and $q\sim p \vcentcolon\iff H_{pq}\ne 0$ or $H_{qp}\ne 0$) restricted to $q \in B_{L+m}(x)$.
By Cauchy--Schwarz on the inner sum (which has at most $\gamma$ terms),
\[
  \left( \sum_{q \sim p} |\psi(q)| \right)^2
  \;\le\;
  \left(\sum_{q \sim p} 1^2\right) \left(\sum_{q \sim p} |\psi(q)|^2\right)
  \;\le\;
  \gamma \sum_{q \sim p} |\psi(q)|^2.
\]
Inserting this into the expression for $R(x)$ gives
\[
  R(x)
  \;\le\; \gamma\, \frac{m^2 M^2}{L^2} \sum_{p \in \Gamma} \sum_{\substack{q \sim p \\ q \in B_{L+m}(x)}} |\psi(q)|^2.
\]
We now swap the summation order. The condition $p \in \Gamma$ and $q \sim p$ is equivalent to $q \in \Gamma$ and $p \sim q$. For a fixed $q$, there are at most $\gamma$ such $p$'s. Thus
\[
  R(x)
  \;\le\; \frac{\gamma^2 m^2 M^2}{L^2} \sum_{q \in B_{L+m}(x)\cap \Gamma} |\psi(q)|^2
  \;=\; \frac{\gamma^2 m^2 M^2}{L^2} \sum_{q \in \Gamma} |\psi(q)|^2 \mathbf{1}_{B_{L+m}(q)}(x),
\]
where we used the symmetry $q \in B_{L+m}(x) \iff x \in B_{L+m}(q)$.
Integrating over $X$ yields
\begin{align*}
  \int_X R(x)\,d\mu(x)
  &\;\le\; \frac{\gamma^2 m^2 M^2}{L^2} \sum_{q \in \Gamma} |\psi(q)|^2 \int_X \mathbf{1}_{B_{L+m}(q)}(x)\,d\mu(x) \\
  &\;=\; \frac{\gamma^2 m^2 M^2}{L^2} \|\psi\|_{\ell^2(\Gamma)}^2 \, V(L+m).
\end{align*}
Finally, we relate $V(L+m)$ to the $L^2$-norm of the tent function. We establish a lower bound by restricting the integration to the ball $B_{L/2}(e)$. For any point $x \in B_{L/2}(e)$, we have $d(x,e) \leq L/2$, which implies
\[
  w_{L,e}(x) \;=\; 1 - \frac{d(x,e)}{L} \;\geq \; 1 - \frac{1}{2} \;=\; \frac{1}{2}.
\]
Squaring and integrating yields
\[
  \|w_{L,e}\|_{L^2(X)}^2
  \;=\; \int_X |w_{L,e}(x)|^2\,d\mu(x)
  \;\ge\; \int_{B_{L/2}(e)} \left(\frac{1}{2}\right)^2\,d\mu(x)
  \;=\; \frac{1}{4} V(L/2).
\]
Using the doubling property iteratively with $k=\lceil\log_2(2(1+m/L))\rceil$, we have $L+m \le 2^k (L/2)$ and thus
\[
  V(L+m) \;\le\; C_{\mathrm{dbl}}^k V(L/2) \;\le\; 4 C_{\mathrm{dbl}}^k \|w_{L,e}\|_{L^2(X)}^2.
\]
Defining $C_{\mathrm{geom}} \coloneqq 4 C_{\mathrm{dbl}}^k$ completes the proof.
\end{proof}

This estimate is the key analytic input in the localisation argument: it
quantifies how well the operator $H$ interacts with localisation at scale $L$,
and will ultimately control the error in passing from global to local
quantities. We now turn to the mechanism that converts this control into
local quasi-modes.

From now on, we restrict to
\[
    L > L_0:=\max\{m,R\},
\]
where $R$ denotes the covering radius of $\Gamma$. For such $L$, the
geometric constant $C_{\mathrm{geom}}$ in Theorem~\ref{thm:commutator}
can be chosen uniformly in $L$, while the
local windows $B_L(x)\cap\Gamma$ are nonempty for all $x\in X$. Consequently,
all subsequent localisation errors are of the form $C_0/L$ with a fixed
constant $C_0$.

\section{From global quasi-modes to local quasi-modes}

We now use the commutator bound to relate global quasi-modes to local ones.
The guiding idea is that quasi-modes (approximate eigenvectors) can be localised without
significant loss, provided $L$ is sufficiently large.

\begin{lemma}[Averaging identity for tent localisations]\label{lem:avg-W}
Let $(X,d,\mu)$ and $\Gamma \subset X$ satisfy the settings of Assumptions~\ref{ass:standing}
and~\ref{ass:operators}, and let $W_{L,x}$ be the multiplication operator on $\ell^2(\Gamma)$
associated with the tent function $w_{L,x}$.
Then, for every $\phi\in\ell^2(\Gamma)$,
\begin{equation}\label{eq:mowing-lawn}
 \int_X \bigl\|W_{L,x}\phi\bigr\|_{\ell^2(\Gamma)}^2\,\mathrm{d}\mu(x)
 \;=\;
 \|\phi\|_{\ell^2(\Gamma)}^2\,\|w_{L,e}\|_{L^2(X)}^2 .
\end{equation}
\end{lemma}

\begin{proof}
We have
\[
 \|W_{L,x}\phi\|_{\ell^2(\Gamma)}^2
 = \sum_{p\in\Gamma} |w_{L,x}(p)|^2\,|\phi(p)|^2.
\]
Since all terms are nonnegative, Tonelli's theorem yields
\[
 \int_{X}\|W_{L,x}\phi\|_{\ell^2(\Gamma)}^2\,\mathrm{d}\mu(x)
 = \sum_{p\in\Gamma}|\phi(p)|^2
   \int_{X}|w_{L,x}(p)|^2\,\mathrm{d}\mu(x).
\]
By left-invariance of the metric,
$ w_{L,x}(p)=w_{L,e}(p^{-1}x). $
Hence, using the substitution $y=p^{-1}x$ and the left-invariance of $\mu$,
\[
  \int_X |w_{L,x}(p)|^2\,\mathrm{d}\mu(x)
  =
  \int_X |w_{L,e}(y)|^2\,\mathrm{d}\mu(y)
  =
  \|w_{L,e}\|_{L^2(X)}^2.
\] Substituting this into the previous identity gives
\eqref{eq:mowing-lawn}.
\end{proof}

Combining the averaging identity with the commutator estimate, we obtain a
localisation principle: global quasi-modes can be converted into local
quasi-modes with only a controlled loss.

\begin{lemma}\label{lem:localise}
Let $(X,d,\mu)$, $\Gamma \subset X$ and $H$ satisfy the settings of Assumptions~\ref{ass:standing}
and~\ref{ass:operators}, and let $\lambda\in\C$.
If $\psi\in\ellTwo(\Gamma)\setminus\{0\}$ satisfies
\[
  \norm{(H-\lambda)\psi}_{\ell^2(\Gamma)}\le \varepsilon\,\norm{\psi}_{\ell^2(\Gamma)}
\]
for some $\varepsilon>0$, then there exists $x\in X$ with $W_{L,x}\psi\neq 0$ such that
\[
  \norm{(H-\lambda)W_{L,x}\psi}_{\ell^2(\Gamma)}
  \;\le\;
  \Bigl(\varepsilon+\frac{C_0}{L}\Bigr)\,\norm{W_{L,x}\psi}_{\ell^2(\Gamma)},
  \qquad
  C_0\coloneqq \gamma mM\sqrt{C_{\mathrm{geom}}}.
\]
Consequently, for some $x\in X$,
\[
  \nuop(H_{L,x}-\lambda)\le \varepsilon+\frac{C_0}{L}.
\]
\end{lemma}

\begin{proof}
Write for each $x\in X$
\[
(H-\lambda )\,W_{L,x}\psi
\;=\; W_{L,x}(H-\lambda )\psi \;-\; [W_{L,x},H]\psi .
\]
Taking the $L^2(X,\mu)$–norm with respect to $x$ and using the triangle inequality
in $L^2$, we obtain
\begin{align*}
 \Bigl(\int_{X} \bigl\|(H - \lambda ) W_{L,x} \psi\bigr\|_{\ell^2(\Gamma)}^2
       \,\mathrm{d}\mu(x)\Bigr)^{1/2}
 &\leq  \Bigl(\int_{X} \bigl\|W_{L,x} (H - \lambda) \psi\bigr\|_{\ell^2(\Gamma)}^2\,
              \mathrm{d}\mu(x)\Bigr)^{1/2} \\
 &\qquad + \Bigl(\int_{X} \bigl\|[W_{L,x}, H] \psi\bigr\|_{\ell^2(\Gamma)}^2\,
                  \mathrm{d}\mu(x)\Bigr)^{1/2}.
\end{align*}

For the first term we apply \eqref{eq:mowing-lawn} with $\phi=(H-\lambda )\psi$:
\begin{align*}
\Bigl(\int_{X} \bigl\|W_{L,x} (H - \lambda ) \psi\bigr\|_{\ell^2(\Gamma)}^2\,\mathrm{d}\mu(x)\Bigr)^{1/2}
&= \|w_{L,e}\|_{L^2(X)} \,\bigl\|(H - \lambda ) \psi\bigr\|_{\ell^2(\Gamma)} \\
&\leq \|w_{L,e}\|_{L^2(X)}\, \varepsilon\, \|\psi\|_{\ell^2(\Gamma)}.
\end{align*}

For the second term we use the commutator estimate from
Theorem~\ref{thm:commutator} (with $\phi=\psi$), 
so
\[
\Bigl(\int_{X} \bigl\|[W_{L,x}, H] \psi\bigr\|_{\ell^2(\Gamma)}^{2}\,\mathrm{d}\mu(x)\Bigr)^{1/2}
\;\le\;
\frac{C_0}{L}\,\|w_{L,e}\|_{L^{2}(X)}\,\|\psi\|_{\ell^2(\Gamma)}.
\]

Combining these bounds, we arrive at
\begin{align*}
\Bigl(\int_{X} \bigl\|(H - \lambda ) W_{L,x} \psi\bigr\|_{\ell^2(\Gamma)}^2\,\mathrm{d}\mu(x)\Bigr)^{1/2}
&\leq \|w_{L,e}\|_{L^2(X)}
      \Bigl(\varepsilon + \frac{C_0}{L}\Bigr)\,\|\psi\|_{\ell^2(\Gamma)}.
\end{align*}

Applying \eqref{eq:mowing-lawn} with $\phi=\psi$ the other way around gives
\[
 \|w_{L,e}\|_{L^2(X)}^2\,\|\psi\|_{\ell^2(\Gamma)}^2 = \int_{X} \bigl\|W_{L,x}\psi\bigr\|_{\ell^2(\Gamma)}^2\,\mathrm{d}\mu(x).
\]
Hence
\begin{equation}
\label{eq:measure argument needed}
\int_{X} \bigl\|(H - \lambda ) W_{L,x} \psi\bigr\|_{\ell^2(\Gamma)}^2\,\mathrm{d}\mu(x)
 \;\leq\;
\int_{X} \Bigl(\varepsilon + \frac{C_0}{L}\Bigr)^{2}
                 \bigl\|W_{L,x}\psi\bigr\|_{\ell^2(\Gamma)}^2\,\mathrm{d}\mu(x).
\end{equation}

By Lemma \eqref{eq:mowing-lawn} with $\|\phi\|_{\ell^2(\Gamma)}^2\,\|w_{L,e}\|_{L^2(X)}^2 > 0$, the set
\[
  E\coloneqq \Bigl\{x\in X:\bigl\|W_{L,x}\psi\bigr\|_{\ell^2(\Gamma)}>0\Bigr\}
\]
has positive $\mu$--measure. Assume for contradiction that
\[
  \bigl\|(H - \lambda ) W_{L,x} \psi\bigr\|_{\ell^2(\Gamma)}^2
  \;>\;
  \Bigl(\varepsilon + \frac{C_0}{L}\Bigr)^{2}\,\bigl\|W_{L,x}\psi\bigr\|_{\ell^2(\Gamma)}^2
  \qquad \text{for all }x\in E .
\]
Since $\bigl\|W_{L,x}\psi\bigr\|_{\ell^2(\Gamma)}=0$ implies $W_{L,x}\psi=0$ and hence
$\bigl\|(H-\lambda)W_{L,x}\psi\bigr\|_{\ell^2(\Gamma)}=0$, the strict inequality above yields
\[
\int_{X} \bigl\|(H - \lambda ) W_{L,x} \psi\bigr\|_{\ell^2(\Gamma)}^2\,\mathrm{d}\mu(x)
 \;>\;
\int_{X} \Bigl(\varepsilon + \frac{C_0}{L}\Bigr)^{2}
                 \bigl\|W_{L,x}\psi\bigr\|_{\ell^2(\Gamma)}^2\,\mathrm{d}\mu(x),
\]
contradicting (\ref{eq:measure argument needed}). Therefore there exists
$x\in X$ with $\|W_{L,x}\psi\|>0$ such that
\[
  \|(H-\lambda)W_{L,x}\psi\|_{\ell^2(\Gamma)}
  \ \le\
  \Bigl( \varepsilon + \frac{C_0}{L} \Bigr)\,\|W_{L,x}\psi\|_{\ell^2(\Gamma)}.
\]

Set $\widetilde{\psi}\coloneqq W_{L,x}\psi\neq 0$. Since $w_{L,x}$ is supported in
$B_L(x)$, we have $\supp(\widetilde{\psi})\subset B_L(x)\cap\Gamma$. By the
definition of the localised operator, $(H_{L,x}-\lambda)$ agrees with
$(H-\lambda)$ on vectors supported in $B_L(x)\cap\Gamma$, and hence
\[
  \|(H_{L,x}-\lambda)\widetilde{\psi}\|_{\ell^2(\Gamma)}
  = \|(H-\lambda)\widetilde{\psi}\|_{\ell^2(\Gamma)}
  \le \Bigl(\varepsilon+\frac{C_0}{L}\Bigr)\,\|\widetilde{\psi}\|_{\ell^2(\Gamma)}.
\]
Taking the infimum over all nonzero vectors supported in $B_L(x)\cap\Gamma$
yields
\[
  \nu\bigl(H_{L,x}-\lambda\bigr)\le \varepsilon+\frac{C_0}{L}.
\] 
\end{proof}

As an immediate consequence, the global lower norm can be controlled from above
by local lower norms.

\begin{corollary}[Global lower norm controls window lower norms]\label{cor:global-local}
Let $(X,d,\mu)$, $\Gamma \subset X$ and $H$ satisfy the settings of Assumptions~\ref{ass:standing}
and~\ref{ass:operators}, then for every $\lambda\in\C$ and $L>L_0$,
\[
  \inf_{x\in X}\nuop(H_{L,x}-\lambda)
  \;\le\;
  \nuop(H-\lambda)+\frac{C_0}{L}, \qquad
  C_0=\gamma mM\sqrt{C_{\mathrm{geom}}}.
\]
\end{corollary}

\begin{proof}
Fix $\lambda\in\C$ and $L>L_0$.  Let $\varepsilon>\nuop(H-\lambda)$.  By the definition
of the lower norm, there exists $\psi\in\ell^2(\Gamma)$ such that
\[
  \|(H-\lambda)\psi\|_{\ell^2(\Gamma)} < \varepsilon \, \|\psi\|_{\ell^2(\Gamma)}.
\]
Applying Lemma~\ref{lem:localise} to this vector $\psi$, we obtain some $x\in X$ with
\[
  \nuop(H_{L,x}-\lambda)\ \le\ \varepsilon+\frac{C_0}{L}.
\]
Taking the infimum over $x$ gives
\[
  \inf_{x\in X}\nuop(H_{L,x}-\lambda)\ \le\ \varepsilon+\frac{C_0}{L}.
\]
Since this holds for every $\varepsilon>\nuop(H-\lambda)$, letting $\varepsilon\downarrow\nuop(H-\lambda)$ yields the claim.
\end{proof}

\begin{remark}
Let $R$ denote the covering radius of $\Gamma$. Then the infimum over $x\in X$
may, at the expense of enlarging the window size by $R$, be restricted to centres
$p\in\Gamma$. Indeed, for every $x\in X$ there exists $p\in\Gamma$ with
$d(x,p)\le R$, and hence
\[
  B_L(x)\cap\Gamma \subset B_{L+R}(p)\cap\Gamma,
  \qquad
  B_{L+m}(x)\cap\Gamma \subset B_{L+m+R}(p)\cap\Gamma.
\]
Therefore every test vector for $H_{L,x}$ is also admissible for $H_{L+R,p}$, so
\[
  \nu(H_{L+R,p}-\lambda)\le \nu(H_{L,x}-\lambda).
\]
Consequently,
\[
  \inf_{p\in\Gamma}\nu(H_{L+R,p}-\lambda)
  \;\le\;
  \inf_{x\in X}\nu(H_{L,x}-\lambda)
  \;\le\;
  \nu(H-\lambda)+\frac{C_0}{L}.
\]
\end{remark}

Thus, the global lower norm can be estimated from above by the infimum of local lower norms with an explicit error term of order $O(1/L)$.
This estimate will be the basis for the spectral and pseudospectral
approximation results in the following sections.

\begin{remark}[The commutator as the sole source of error]
\label{rem:commutator-role}
The commutator estimate in Theorem~\ref{thm:commutator} is the central
technical ingredient of the entire argument. It isolates the only mechanism
by which localisation can fail.

Indeed, if $H$ commuted with the localisation operators $W_{L,x}$, then one
would have
\[
  (H-\lambda)W_{L,x}\psi = W_{L,x}(H-\lambda)\psi,
\]
and Lemma~\ref{lem:localise} would follow directly from
Lemma~\ref{lem:avg-W}, without any loss term. In this idealised situation
(for instance for diagonal operators), global quasi-modes would localise
exactly, and all subsequent results would hold with $C_0=0$.

Thus, the commutator $[W_{L,x},H]$ measures precisely the defect of
commutation, and the estimate in Theorem~\ref{thm:commutator} quantifies
how much error is introduced when passing from global to local information.
The explicit $O(1/L)$ term appearing throughout the paper is entirely due to
this effect.
\end{remark}

\section{Local pseudospectra and convergence}

\medskip
We now pass from the commutator estimate to quantitative pseudospectral
inclusions for the local restrictions $H_{L,x}$. The guiding principle is that
$\sigma_\varepsilon(H)$ can be characterised in terms of lower norms of
$H-\lambda$ and $H^*-\overline{\lambda}$, and our local quantities approximate
these two global objects simultaneously.

In the rectangular setting, one therefore has to treat the contributions of
$H$ and $H^*$ separately (see Section~\ref{sec:rectangular-pseudo}), and combine
them only at the level of lower norms.

\begin{theorem}[Window pseudospectra and convergence]\label{thm:window-pseudo-convergence}
Let $(X,d,\mu)$, $\Gamma \subset X$ and $H$ satisfy the settings of Assumptions~\ref{ass:standing}
and~\ref{ass:operators} and suppose that the commutator estimate (\ref{eq:commutator}) holds for
both $H$ and $H^*$ with the same constant $C_0>0$.  Let $H_{L,x}$ denote the rectangular window restriction introduced in Section~\ref{sec:band-and-windows}. For $\varepsilon > 0$ and $\eta \ge 0$ define the window
pseudospectra
\begin{align*}
  \gamma_{L,\varepsilon}(H)
  &:= \Bigl\{\lambda\in\C:\;
        \inf_{x\in X}\,
        \min\bigl\{
          \nu(H_{L,x}-\lambda),\,
          \nu((H^*)_{L,x}-\overline{\lambda})
        \bigr\}
        < \varepsilon
      \Bigr\},\\
  \Gamma_{L,\eta}(H)
  &:= \Bigl\{\lambda\in\C:\;
        \inf_{x\in X}\,
        \min\bigl\{
          \nu(H_{L,x}-\lambda),\,
          \nu((H^*)_{L,x}-\overline{\lambda})
        \bigr\}
        \le \eta
      \Bigr\}. 
\end{align*}
Then the following holds for all $L>L_0$:
\begin{enumerate}
\item[(a)] \textup{(Inclusion bounds)} For every $\varepsilon>0$ and every $\eta\ge 0$ one has
\[
  \gamma_{L,\varepsilon}(H)\subset \sigma_\varepsilon(H)\subset
  \gamma_{L,\varepsilon+C_0/L}(H),
  \qquad
  \Gamma_{L,\eta}(H)\subset \Sigma_\eta(H)\subset
  \Gamma_{L,\eta+C_0/L}(H).
\]

\item[(b)] \textup{(Hausdorff convergence)} For every fixed $\varepsilon>0$ and every fixed $\eta>0$,
\[
  \gamma_{L,\varepsilon}(H)\xrightarrow[d_H]{L\to\infty}\sigma_\varepsilon(H),
  \qquad
  \Gamma_{L,\eta}(H)\xrightarrow[d_H]{L\to\infty}\Sigma_\eta(H).
\]
Moreover,
\[
  \Gamma_{L,C_0/L}(H)\xrightarrow[d_H]{L\to\infty}\Sigma_0(H)=\sigma(H).
\]
\end{enumerate}
\end{theorem}

\begin{proof}
(a) Fix $\lambda\in\C$.  By Corollary~\ref{cor:global-local} applied to $H-\lambda$
and to $(H-\lambda)^*=H^*-\overline{\lambda}$ we obtain
\[
  \inf_{x\in X}\nu(H_{L,x}-\lambda)
  \;\le\; \nu(H-\lambda)+\frac{C_0}{L},
  \qquad
  \inf_{x\in X}\nu((H^*)_{L,x}-\overline{\lambda})
  \;\le\; \nu(H^*-\overline{\lambda})+\frac{C_0}{L}.
\]
Taking the minimum yields
\[
  \inf_{x\in X}\min\bigl\{\nu(H_{L,x}-\lambda),\,\nu((H^*)_{L,x}-\overline{\lambda})\bigr\}
  \;\le\; \nu_{comb}(H-\lambda)+\frac{C_0}{L}.
\]
Conversely, for every $x\in X$ we have the monotonicity
$\nu(H-\lambda)\le \nu(H_{L,x}-\lambda)$ and
$\nu(H^*-\overline{\lambda})\le \nu((H^*)_{L,x}-\overline{\lambda})$, hence
\[
  \nu_{comb}(H-\lambda)
  \;\le\;
  \inf_{x\in X}\min\bigl\{\nu(H_{L,x}-\lambda),\,\nu((H^*)_{L,x}-\overline{\lambda})\bigr\}.
\]
Thus the window quantity sandwiches $\nu_{comb}(H-\lambda)$ up to an error
$C_0/L$, and taking strict/closed sublevel sets gives the inclusions.

(b) 
If $\eta>0$, choose $L$ so large that $\eta-\frac{C_0}{L}>0$. Then the inclusions
from part~(a) give
\[
  \Sigma_{\eta-\frac{C_0}{L}}(H)
  \subset
  \Gamma_{L,\eta}(H)
  \subset
  \Sigma_{\eta}(H)
  \subset
  \Gamma_{L,\eta+\frac{C_0}{L}}(H)
  \subset
  \Sigma_{\eta+\frac{C_0}{L}}(H).
\]
Since $\eta\mapsto \Sigma_\eta(H)$ is Hausdorff--continuous on $\ell^2$
(Section~\ref{sec:hausdorff-globevnik}), we have
\[
  \Sigma_{\eta\pm \frac{C_0}{L}}(H)\xrightarrow[d_H]{}\Sigma_\eta(H)
  \qquad (L\to\infty).
\]
Hence a standard sandwich argument yields
\[
  \Gamma_{L,\eta}(H)\xrightarrow[d_H]{}\Sigma_\eta(H).
\]

If $\eta=0$, the left-hand side of the above chain is no longer available. In this case
one uses only the right-hand side of the inclusions in part~(a), namely
\[
  \Sigma_0(H)
  \subset
  \Gamma_{L,\frac{C_0}{L}}(H)
  \subset
  \Sigma_{\frac{C_0}{L}}(H).
\]
Again, Hausdorff continuity implies
\[
  \Sigma_{\frac{C_0}{L}}(H)\xrightarrow[d_H]{}\Sigma_0(H),
\]
and therefore
\[
  \Gamma_{L,\frac{C_0}{L}}(H)\xrightarrow[d_H]{}\Sigma_0(H).
\]

The statement for $\gamma_{L,\varepsilon}(H)$ and $\sigma_\varepsilon(H)$ with
$\varepsilon>0$ follows from a standard sandwich argument and
\[
  \Sigma_\varepsilon(H)=\overline{\sigma_\varepsilon(H)}
\]
(the Globevnik property) together with the fact that $d_H$ depends only on closures.
\end{proof}

\begin{remark}[Union representation]\label{rem:gamma-union}
For $\varepsilon>0$ and $L>L_0$ one has
\[
  \gamma_{L,\varepsilon}(H)
  \;=\;
  \bigcup_{x\in X}
  \Bigl(
    \sigma_\varepsilon(H_{L,x})
    \,\cup\,
    \sigma_\varepsilon\bigl((H^*)_{L,x}\bigr)^{*}
  \Bigr),
  \qquad
  S^*:=\{\overline{z}:z\in S\}.
\]
Indeed, the defining condition for $\gamma_{L,\varepsilon}(H)$ is an infimum
over $x$ being strictly smaller than $\varepsilon$, which is equivalent to the
existence of some $x$ for which one of $
          \nu(H_{L,x}-\lambda)$ and $
          \nu((H^*)_{L,x}-\overline{\lambda})$ is less than $\varepsilon$.
\end{remark}

\begin{remark}[Optimality of the $1/L$ rate]\label{rem:optimality}
The $1/L$ scaling is dictated by Lipschitz localisation: any cutoff $f$ supported in $B_L(x_0)$ with $0\le f\le 1$ and $f(x_0)=1$ must satisfy $\Lip(f)\ge 1/L$ whenever the metric realises distances at scale $L$. Thus commutator bounds based on compactly supported Lipschitz cutoffs cannot, in general, decay faster than $O(1/L)$.
\end{remark}

\subsection{Local gap estimates and spectral consequences}

We now summarise the consequences of the localisation estimates for the
(global) lower norm and the spectrum for finite-interaction-range operators on spaces satisfying Assumption \ref{ass:standing}. Let $\Gamma$ and $H$ as in Assumption \ref{ass:operators} and let $C_0:=\gamma mM\sqrt{C_{\mathrm{geom}}}$ denote the constant from Lemma \ref{lem:localise}.

For $L>L_0$ and $\lambda\in\C$ define
\[
  \alpha_L(\lambda)
  := \inf_{x\in X}\nuop(H_{L,x}-\lambda)\qquad
  \alpha_{L,comb}(\lambda)
  := \inf_{x\in X}\min\bigl\{\nu(H_{L,x}-\lambda),\,\nu((H^*)_{L,x}-\overline{\lambda})\bigr\}.
\]

\begin{corollary}[Local lower norm controls distance to spectrum]
\label{thm:local-gap-general}
For every $\lambda\in\C$ and $L>L_0$ one has
\[
  \dist(\lambda,\sigma(H))
  \;\ge\;
  \nu_{comb}(H-\lambda)
  \;\ge\;
  \alpha_{L, comb}(\lambda)-\frac{C_0}{L}.
\]
\end{corollary}

\begin{proof}
The first inequality is standard \cite{Lindner2006, RabinovichRochSilbermann2004}. The second follows from
Corollary~\ref{cor:global-local}.
\end{proof}
\noindent
If $H$ is normal, then the first inequality becomes an equality and $\nu_{comb}(H-\lambda)=\nu(H-\lambda)$, i.e. $\dist(\lambda,\sigma(H))=\nuop(H-\lambda) \geq \alpha_{L}(\lambda)-\frac{C_0}{L}$. 

In that case, a strictly positive lower bound on
$\alpha_L(\lambda)-C_0/L$ yields an explicit lower bound on
$\dist(\lambda,\sigma(H))$, and hence guarantees a spectral gap.

\begin{corollary}[Local gap test]
\label{cor:local-gap}
Let $\lambda\in\C$ and $L>L_0$.

\begin{enumerate}
\item If $\delta:=\alpha_{L,\mathrm{comb}}(\lambda)-\frac{C_0}{L}>0, $
then
\[
  B_\delta(\lambda)\cap\sigma(H)=\varnothing.
\]

\item If, in addition, $H$ is normal, then the combined local lower norm may be
replaced by the ordinary local lower norm. Thus, if
$ \delta:=\alpha_L(\lambda)-\frac{C_0}{L}>0,$
then
\[
  B_\delta^\circ(\lambda)\cap\sigma(H)=\varnothing.
\]
\end{enumerate}
\end{corollary}

\begin{proof}
The first assertion follows directly from
Corollary~\ref{thm:local-gap-general}. If $H$ is normal, then
\[
  \dist(\lambda,\sigma(H))=\nu(H-\lambda)
  \ge \alpha_L(\lambda)-\frac{C_0}{L},
\]
where the inequality follows from Corollary~\ref{cor:global-local}. This proves the
second assertion.
\end{proof}

\medskip

\noindent
In the self-adjoint case, $\sigma(H)\subset\R$, and the previous result yields
open intervals free of spectrum:
if $\lambda\in\R$ and $\delta>0$ as above, then
\[
  (\lambda-\delta,\lambda+\delta)\cap\sigma(H)=\varnothing.
\]

\begin{corollary}[Sampling the spectrum from local lower norms in the normal case]
\label{cor:normal-sampling-grid}
Let $\Gamma$ and $H$ satisfy Assumption \ref{ass:operators} and assume that $H$ is normal. 
For fixed $\delta>0$, choose $L$ so large that
\[
  \frac{2C_0}{L}<\delta.
\]
Set
\[
  h_L:=\frac{\sqrt2\,C_0}{L}
  \qquad\text{and}\qquad
  R>\|H\|+\frac{C_0}{L},
\]
and consider the finite grid
\[
  G_{L,R}
  :=\bigl\{a+ib\in\C:\ a,b\in h_L\Z,\ |a+ib|\le R\bigr\}.
\]
Define
\[
  \Lambda_L(\delta)
  :=\bigl\{\lambda\in G_{L,R}:\ \alpha_L(\lambda)<\delta\bigr\}.
\]
Then
\[
  d_H\bigl(\sigma(H),\Lambda_L(\delta)\bigr)<\delta.
\]
In particular, the spectrum of a normal finite--interaction--range operator can be
approximated arbitrarily well in Hausdorff distance by finitely many sample
points detected from local lower norms.
\end{corollary}

\begin{proof}
Let $\lambda\in\sigma(H)$. By construction of the grid, there exists
$\lambda'\in G_{L,R}$ with
\[
  |\lambda-\lambda'|\le \frac{h_L}{\sqrt2}=\frac{C_0}{L}.
\]
Since $\nuop$ is 1-Lipschitz and $\nuop(H-\lambda)=0$, we have
\[
  \nuop(H-\lambda')\le \nuop(H-\lambda) + |\lambda-\lambda'|\le \frac{C_0}{L}.
\]
By Corollary~\ref{cor:global-local},
\[
  \alpha_L(\lambda')
  \le
  \nuop(H-\lambda')+\frac{C_0}{L}
  \le
  \frac{2C_0}{L}
  <\delta.
\]
Thus $\lambda'\in\Lambda_L(\delta)$, and every spectral point lies within
distance at most $C_0/L<\delta$ of $\Lambda_L(\delta)$.

Conversely, let $\lambda\in\Lambda_L(\delta)$. Then $\alpha_L(\lambda)<\delta$,
and by monotonicity of local lower norms,
\[
  \nuop(H-\lambda)\le \alpha_L(\lambda)<\delta.
\]
Since $H$ is normal,
\[
  \dist(\lambda,\sigma(H))=\nuop(H-\lambda)<\delta.
\]
Hence every point of $\Lambda_L(\delta)$ lies within distance $<\delta$ of
$\sigma(H)$. The two estimates together imply
\[
  d_H\bigl(\sigma(H),\Lambda_L(\delta)\bigr)<\delta.
\]
\end{proof}

\begin{remark}[Practical procedure]
\label{rem:sampling-procedure}
Corollary~\ref{cor:normal-sampling-grid} yields the following finite procedure:
\begin{enumerate}
\item choose $\delta>0$ and $L$ as in the corollary;
\item evaluate $\alpha_L(\lambda)$ for $\lambda\in G_{L,R}$;
\item retain those grid points for which $\alpha_L(\lambda)<\delta$.
\end{enumerate}
Then $\Lambda_L(\delta)$ is a rigorous $\delta$--Hausdorff approximation of $\sigma(H)$. Under finite local complexity, the values $\alpha_L(\lambda)$ are determined by finitely many local lower norms, so that the whole procedure can in fact be carried out in finite time.
\end{remark}

The normal case relied on the identity
$\nu(H-\lambda)=\dist(\lambda,\sigma(H))$,
which allows one to recover spectral information directly from lower norms. In the non-normal case, this mechanism breaks down: although $\nu_{\mathrm{comb}}(H-\lambda)$ still detects invertibility, it no longer determines the distance from $\lambda$ to the spectrum. Since the spectrum may then be highly unstable under small perturbations, the natural object to approximate is the pseudospectrum.

This also changes the algorithmic picture. In general, the required grid size is no longer available \emph{a priori}; instead, one must refine the grid until successive local pseudospectral approximants are sufficiently close. The continuity of the pseudospectrum in $\varepsilon$ guarantees that this stage is eventually reached, and the resulting closeness condition provides a stopping criterion. Thus the procedure remains deterministic, although the necessary grid resolution is only identified during the computation. This leads to the following stopping criterion.

\begin{theorem}[Approximation of the pseudospectrum in the non-normal case]
\label{thm:pseudospectrum-nonnormal}
Let $H$ be as in Assumption \ref{ass:operators}, and fix $\varepsilon>0$ and $\delta>0$. For $L>L_0$, set
\[
h_L\coloneqq \frac{C_0}{L} \text{ and } R > \|H\|+\varepsilon+\frac{C_0}{L}.
\]
Let $G_{L,R}\subset\C$ be the grid with mesh size $h_L$ (as in Corollary \ref{cor:normal-sampling-grid}), and define
\[
\Lambda_\eta^{(h_L)}:=\{\mu\in G_{L,R}:\alpha_{L,comb}(\mu) < \eta\},
\qquad \eta > 0.
\]
Choose $L>L_0$ so large that
\[
h_L<\min \{ \frac{\delta}{4}, \frac{\varepsilon}{2} \}
\qquad\text{and}\qquad
d_H\bigl(\Lambda_{\varepsilon+2h_L}^{(h_L)},\Lambda_\varepsilon^{(h_L)}\bigr)<\frac{\delta}{2}.
\]
Then
\[
d_H\bigl(\sigma_\varepsilon(H),\Lambda_\varepsilon^{(h_L)}\bigr)<\delta.
\]
\end{theorem}

\begin{proof}
Since $\Lambda_\varepsilon^{(h_L)}\subset \sigma_\varepsilon(H)$, it suffices to show that
\[
\sup_{\lambda\in\sigma_\varepsilon(H)}
\dist\bigl(\lambda,\Lambda_\varepsilon^{(h_L)}\bigr)<\delta.
\]

The proof has two ingredients. First, one can indeed choose $L$ so large that
\[
d_H\bigl(\Lambda_{\varepsilon+2h_L}^{(h_L)},\Lambda_\varepsilon^{(h_L)}\bigr)<\frac{\delta}{2}.
\]
In the Globevnik case, the map
\[
\eta\longmapsto \sigma_\eta(H)
\]
is continuous with respect to the Hausdorff distance. So one can choose $L>L_0$ such that
\[
d_H\bigl(\sigma_{\varepsilon+2h_L}(H),\sigma_{\varepsilon-2h_L}(H)\bigr) < \frac{\delta}{4}.
\]
Let $\eta\in \Lambda_{\varepsilon+2h_L}^{(h_L)}$. Then $\eta\in \sigma_{\varepsilon+2h_L}(H)$. Hence we can choose
\[
\eta'\in \sigma_{\varepsilon-2h_L}(H)
\]
such that
\[
|\eta-\eta'|
\le d_H\bigl(\sigma_{\varepsilon+2h_L}(H),\sigma_{\varepsilon-2h_L}(H)\bigr) < \frac{\delta}{4}
\]
Since $G_L$ has mesh size $h_L$, there exists $\eta''\in G_L$ with
\[
|\eta'-\eta''|< h_L < \frac{\delta}{4}.
\]
Using the $1$-Lipschitz continuity of the combined lower norm, we obtain
\[
\nu_{comb}(H-\eta'')
\le \nu_{comb}(H-\eta')+|\eta'-\eta''|
< \varepsilon-2h_L+|\eta'-\eta''|
< \varepsilon-h_L.
\]
Hence
\[
\alpha_{L,comb}(\eta'')
\le \nu_{comb}(H-\eta'')+\frac{C_0}{L}
< \varepsilon-h_L+h_L
= \varepsilon,
\]
so that $\eta''\in \Lambda_\varepsilon^{(h_L)}$. Moreover,
\[
|\eta-\eta''|
\le |\eta-\eta'|+|\eta'-\eta''|
< \frac{\delta}{4}+h_L < \frac{\delta}{2}.
\]
\noindent
Thus we may choose $L$ so large that
\[
h_L<\frac{\delta}{4}
\qquad\text{and}\qquad
d_H\bigl(\Lambda_{\varepsilon+2h_L}^{(h_L)},\Lambda_\varepsilon^{(h_L)}\bigr)<\frac{\delta}{2}.
\]

Second, let $\lambda\in \sigma_\varepsilon(H)$. Since $G_L$ has mesh size $h_L$, there exists $\lambda'\in G_L$ such that
\[
|\lambda-\lambda'|<h_L.
\]
Therefore,
\[
\alpha_{L,comb}(\lambda')
\le \nu_{comb}(H-\lambda')+\frac{C_0}{L}
\le \nu_{comb}(H-\lambda)+|\lambda-\lambda'|+\frac{C_0}{L}
< \varepsilon+2h_L.
\]
Hence
\[
\lambda'\in \Lambda_{\varepsilon+2h_L}^{(h_L)}.
\]
It follows that
\[
\dist\bigl(\lambda,\Lambda_{\varepsilon+2h_L}^{(h_L)}\bigr)<h_L<\frac{\delta}{4}.
\]

Combining this with the choice of $L$, we obtain
\[
\dist\bigl(\lambda,\Lambda_\varepsilon^{(h_L)}\bigr)
\le
\dist\bigl(\lambda,\Lambda_{\varepsilon+2h_L}^{(h_L)}\bigr)
+
d_H\bigl(\Lambda_{\varepsilon+2h_L}^{(h_L)},\Lambda_\varepsilon^{(h_L)}\bigr)
<
\frac{\delta}{4}+\frac{\delta}{2}
<
\delta.
\]
Taking the supremum over all $\lambda\in \sigma_\varepsilon(H)$ yields
\[
d_H\bigl(\sigma_\varepsilon(H),\Lambda_\varepsilon^{(h_L)}\bigr)<\delta,
\]
as claimed.
\end{proof}

\begin{remark}[Practical procedure in the non-normal case]
\label{rem:pseudo-sampling-procedure}
Theorem~\ref{thm:pseudospectrum-nonnormal} suggests the following deterministic
procedure for approximating $\sigma_\varepsilon(H)$:
\begin{enumerate}
\item fix $\varepsilon>0$ and $\delta>0$, and choose an initial scale $L>L_0$;
\item form the grid $G_{L,R}$ with mesh size $h_L=C_0/L$ and compute
\[
  \Lambda_\varepsilon^{(h_L)}
  =\{\mu\in G_{L,R}:\alpha_{L,\mathrm{comb}}(\mu) < \varepsilon\},
  \qquad
  \Lambda_{\varepsilon+2h_L}^{(h_L)}
  =\{\mu\in G_{L,R}:\alpha_{L,\mathrm{comb}}(\mu) < \varepsilon+2h_L\};
\]
\item check whether
\[
  h_L<\frac{\delta}{4}
  \qquad\text{and}\qquad
  d_H\bigl(\Lambda_{\varepsilon+2h_L}^{(h_L)},\Lambda_\varepsilon^{(h_L)}\bigr)
  <\frac{\delta}{2};
\]
\item if the stopping condition (3.) is satisfied, output $\Lambda_\varepsilon^{(h_L)}$;
otherwise increase $L$ (equivalently, refine the grid) and repeat.
\end{enumerate}
By continuity of the pseudospectrum in $\varepsilon$, this procedure eventually
stops, and the output is then a rigorous $\delta$--Hausdorff approximation of
$\sigma_\varepsilon(H)$. Thus, although the required grid size is not known
\emph{a priori}, it is found after finitely many refinement steps. Under finite local
complexity, each evaluation of $\alpha_{L,\mathrm{comb}}$ is determined by finitely
many local combined lower norms, so that every stage of the procedure can be carried
out in finite time.
\end{remark}

\section{Extensions and outlook}

\subsection{Extension to non-relatively dense sets}

As already indicated, the present framework can likely be extended to uniformly discrete sets
$\Gamma\subset X$ that are not relatively dense and may therefore have arbitrarily large gaps.
In the arguments above, relative denseness is mainly used to ensure that, for sufficiently large $L$,
the local windows $B_L(x)\cap\Gamma$ are nonempty for all $x\in X$, so that the corresponding
finite sections are defined on nontrivial spaces. There are at least two ways to deal with this issue.

A first possibility is to keep the present setup and modify only the treatment of empty windows. More precisely, in all statements and proofs involving finite sections --- in particular in Theorem~\ref{thm:commutator}, Lemma~\ref{lem:localise}, and Corollary~\ref{cor:global-local} --- one may restrict the integration and the infimum over $x$ to those centres for which the finite section exists, i.e. to
\[
  X_L(\Gamma):=\{x\in X:\, B_L(x)\cap\Gamma\neq\emptyset\}.
\]
Equivalently, one may declare the lower norm of a non-existing finite section to be equal to $+\infty$. With this modification, the corresponding arguments do not appear to require any essential change, since the estimates are local and empty windows simply do not contribute to the minimisation. In this way, Theorem~\ref{thm:window-pseudo-convergence} extends to uniformly discrete sets $\Gamma\subset X$ without relative denseness.

A second possibility is to embed $\Gamma$ into a larger relatively dense set $\widetilde{\Gamma}\supset\Gamma$
and to extend the operator accordingly. Writing
\[
  \ell^2(\widetilde{\Gamma})
  =
  \ell^2(\Gamma)\oplus \ell^2(\widetilde{\Gamma}\setminus\Gamma),
\]
one may define
\[
  \widetilde{H}:=H\oplus C\,I_{\ell^2(\widetilde{\Gamma}\setminus\Gamma)},
\]
where $C\in\C$ is chosen outside the spectral region of interest. Then
\[
  \sigma(\widetilde{H})=\sigma(H)\cup\{C\},
\]
and similarly,
\[
  \sigma_\varepsilon(\widetilde{H})
  =
  \sigma_\varepsilon(H)\cup \overline{B_\varepsilon(C)}.
\]
Thus one recovers the relatively dense situation at the price of adding a completely explicit and
easily removable spectral component.

\subsection{Possible extensions: Banach-valued functions.}
A further natural extension is to replace the scalar-valued space $\ell^2(\Gamma)$ by the Banach-valued sequence space
\[
  \ell^2(\Gamma,E)
  :=
  \Bigl\{u:\Gamma\to E:\ \sum_{p\in\Gamma}\|u(p)\|_E^2<\infty\Bigr\},
\]
where $E$ is a Banach space, equipped with the norm
\[
  \|u\|_{\ell^2(\Gamma,E)}
  :=
  \Bigl(\sum_{p\in\Gamma}\|u(p)\|_E^2\Bigr)^{1/2}.
\]
In this setting, band operators are described in the same way as above, except that the matrix entries become operator-valued,
\[
  H_{pq}\in\mathcal L(E),
\]
with
\[
  \|H_{pq}\|_{\mathcal L(E)}\le M
  \qquad\text{and}\qquad
  d(p,q)>m \Longrightarrow H_{pq}=0.
\]
The action of $H$ is then given by
\[
  (Hu)(p)=\sum_{q\in\Gamma}H_{pq}u(q).
\]

At the level of lower norms, the whole framework extends in a rather direct way. Indeed, for any bounded operator $T$ on a Banach space one may define
\[
  \nu(T):=\inf_{\|u\|=1}\|Tu\|,
\]
so in particular $\nu(H-\lambda)$ is well-defined on $\ell^2(\Gamma,E)$. Likewise, the local lower norms of the finite sections remain meaningful without modification.

Moreover, the localisation operators are unchanged, since the tent functions remain scalar-valued:
\[
  (W_{L,x}u)(p):=w_{L,x}(p)\,u(p).
\]
Similarly, the rectangular finite sections are defined in exactly the same way,
\[
  H_{L,x}=P_{L+m,x}HP_{L,x}^*
  \;:\;
  \ell^2\bigl(B_L(x)\cap\Gamma,E\bigr)
  \longrightarrow
  \ell^2\bigl(B_{L+m}(x)\cap\Gamma,E\bigr),
\]
and all support considerations remain purely geometric. Thus the main points to check are again Theorem~\ref{thm:commutator}, Lemma~\ref{lem:localise}, and Corollary~\ref{cor:global-local}.

These arguments, however, depend only on the geometry of $\Gamma$, the support properties of the finite sections, the Lipschitz bounds for the tent functions, and basic norm estimates such as
\[
  \|H_{pq}u(q)\|_E
  \le
  \|H_{pq}\|_{\mathcal L(E)}\,\|u(q)\|_E.
\]
In particular, the use of the Cauchy--Schwarz inequality remains fully available, since it is applied to the scalar sequence
\[
  \bigl(\|u(p)\|_E\bigr)_{p\in\Gamma}\in \ell^2(\Gamma),
\]
and thus relies only on the Hilbert-space structure of the outer $\ell^2$-summation over $\Gamma$, not on any Hilbert-space structure of the coefficient space $E$. For this reason, the localisation procedure and the resulting approximation results for lower norms carry over to the Banach-valued setting with essentially no change. The only point where a minor modification is needed is the passage from the lower norm to a combined lower norm: in the Banach-valued setting, the Hilbert-space adjoint $H^*$ has to be replaced by the Banach-space adjoint $H'$, which under the natural identification $\ell^2(\Gamma,E)'\cong \ell^2(\Gamma,E')$ is given by
\[
  (H'\varphi)(q)=\sum_{p\in\Gamma} H_{pq}'\,\varphi(p).
\]
Thus the natural analogue of the combined lower norm is expected to involve both $\nu(H-\lambda)$ and $\nu(H'-\lambda)$.

\subsection{Possible extensions: \texorpdfstring{$\ell^p$}{lp}-spaces.}
The arguments of Theorem~\ref{thm:commutator}, Lemma~\ref{lem:avg-W},
Lemma~\ref{lem:localise}, and Corollary~\ref{cor:global-local} are not specific to
$p=2$. For every $1 \leq p<\infty$, one may work on $\ell^p(\Gamma)$ and replace the
use of Cauchy--Schwarz in the proof of Theorem~\ref{thm:commutator} by H\"older's
inequality. Since each inner sum contains at most $\gamma$ terms, one has
\[
  \Bigl(\sum_{q\sim p} |\psi(q)|\Bigr)^p
  \le
  \gamma^{p-1}\sum_{q\sim p} |\psi(q)|^p,
\]
which leads, after the same summation and integration steps as before, to
\[
  \int_X \|[W_{L,x},H]\psi\|_{\ell^p(\Gamma)}^p\,d\mu(x)
  \le
  \frac{\gamma^p m^p M^p C_{\mathrm{geom},p}}{L^p}\,
  \|\psi\|_{\ell^p(\Gamma)}^p\,\|w_{L,e}\|_{L^p(X)}^p,
\]
where
\[
  C_{\mathrm{geom},p}:=2^p C_{\mathrm{dbl}}^{\,k},
  \qquad
  k=\Bigl\lceil\log_2\bigl(2(1+m/L)\bigr)\Bigr\rceil.
\]

Likewise, Lemma~\ref{lem:avg-W} extends verbatim to
\[
  \int_X \|W_{L,x}\phi\|_{\ell^p(\Gamma)}^p\,d\mu(x)
  =
  \|\phi\|_{\ell^p(\Gamma)}^p\,\|w_{L,e}\|_{L^p(X)}^p,
\]
by the same Tonelli argument and left-invariance of $d$ and $\mu$. Once these two
ingredients are available, Lemma~\ref{lem:localise} and
Corollary~\ref{cor:global-local} follow by the same reasoning as before, using the
triangle inequality in $L^p(X;\ell^p(\Gamma))$. Thus the lower-norm based part of the theory admits a direct \(\ell^p\)-analogue for every \(1 \leq p<\infty\), with appropriately modified constants.

The only genuinely new point arises later when one passes to combined lower norms,
where for $p\neq2$ the Hilbert-space adjoint must be replaced by the Banach adjoint on
the dual \(\ell^q\)-space, with \(1/p+1/q=1\).

An analogous extension also holds for $\ell^\infty(\Gamma)$. In that case, the averaging over
$X$ is replaced by direct estimates in the supremum norm, and the analogue of
Lemma~\ref{lem:avg-W} becomes
\[
  \sup_{x\in X}\|W_{L,x}\phi\|_{\ell^\infty(\Gamma)}
  = \|\phi\|_{\ell^\infty(\Gamma)},
\]
since $0\le w_{L,x}\le 1$ and $w_{L,p}(p)=1$ for $p\in\Gamma$. Moreover, using the Lipschitz
bound for the tent functions, one obtains
\[
  \|[W_{L,x},H]\psi\|_{\ell^\infty(\Gamma)}
  \le
  \frac{\gamma mM}{L}\,\|\psi\|_{\ell^\infty(\Gamma)},
\]
Thus the localisation error is again of order $O(1/L)$, but the argument in Theorem \ref{thm:commutator} no longer uses
$L^p(X)$-averaging.

\bibliographystyle{amsplain}
\bibliography{literature}

\end{document}